\documentclass[reqno]{amsart}
\usepackage{xcolor}
\definecolor{rltblue}{rgb}{0,0,0.4}
\definecolor{drkred}{rgb}{0.6,0,0}
\definecolor{drkgreen}{rgb}{0,0.4,0}
\usepackage[colorlinks=true, urlcolor=rltblue, citecolor=drkgreen, linkcolor=drkgreen] {hyperref}
\usepackage{orcidlink}
\usepackage{graphicx} 
\usepackage[margin=1in]{geometry}
\usepackage{comment}

\usepackage{thmtools} 
\usepackage{thm-restate}
\usepackage{float}
\usepackage{amssymb, amsthm, amsmath, amsfonts}
\usepackage{thmtools}
\usepackage{mathrsfs}
\usepackage{latexsym}
\usepackage{datetime}
\usepackage{hyperref}
\usepackage{cleveref}
\usepackage{bm}
\usepackage{tikz}
\usepackage{enumerate}
\usepackage{enumitem}

\newtheorem{theorem}{Theorem}
\newtheorem{lemma}[theorem]{Lemma}
\newtheorem{corollary}[theorem]{Corollary}
\newtheorem{proposition}[theorem]{Proposition}

\newtheorem{definition}[theorem]{Definition}

\newtheorem{remark}[theorem]{Remark}

\newcommand{\A}{\mathcal{A}}
\newcommand{\B}{\mathcal{B}}
\newcommand{\C}{\mathcal{C}}
\newcommand{\D}{\mathcal{D}}
\newcommand{\Std}{\mathcal{S}}
\newcommand{\PR}{\mathcal{PR}}
\newcommand{\N}{\mathbb N}
\newcommand{\Nat}{\mathbb N}
\newcommand{\LL}{\mathcal{L}}

\newcommand{\Succ}{S}
\newcommand{\Pred}{\operatorname{Pred}}
\newcommand{\code}[1]{\ulcorner{#1}\urcorner}

\newcommand{\pair}[1]{\langle{#1}\rangle}

\newtheorem{convention}[theorem]{Convention}
\newtheorem{observation}[theorem]{Observation}

\usepackage[backend=biber,url=false,maxbibnames=10]{biblatex}
\author[Bazhenov]{Nikolay Bazhenov\orcidlink{0000-0002-5834-2770}}
\author[Georgiev]{Ivan Georgiev\orcidlink{0000-0002-2475-6086}}
\author[Kalociński]{Dariusz Kalociński\orcidlink{0000-0002-3044-525X}}
\author[Vatev]{Stefan Vatev\orcidlink{0000-0001-5719-1467}}
\author[Wrocławski]{Michał Wrocławski\orcidlink{0000-0003-2679-7321}}

\address[Bazhenov]{Novosibirsk State University, Novosibirsk, Russia and Innopolis University, Innopolis, Russia}
\email{\href{mailto:bazhenov@math.nsc.ru}{bazhenov@math.nsc.ru}}
\urladdr{\url{https://bazhenov.droppages.site}}
\address[Georgiev]{Faculty of Mathematics and Informatics, Sofia University, Bulgaria}
\email{\href{mailto:ivandg@fmi.uni-sofia.bg}{ivandg@fmi.uni-sofia.bg}}
\address[Kalociński]{Institute of Computer Science, Polish Academy of Sciences, Poland}
\email{\href{mailto:dariusz.kalocinski@gmail.com}{dariusz.kalocinski@gmail.com}}
\urladdr{\url{https://www.dariuszkalocinski.com}}
\address[Vatev]{Faculty of Mathematics and Informatics, Sofia University, Bulgaria}
\email{\href{mailto:stefanv@fmi.uni-sofia.bg}{stefanv@fmi.uni-sofia.bg}}
\urladdr{\url{https://fmi.uni-sofia.bg/fmi/logic/stefanv/}}
\address[Wrocławski]{Faculty of Philosophy, University of Warsaw, Poland}
\email{\href{mailto:m.wroclawski@uw.edu.pl}{m.wroclawski@uw.edu.pl}}


\thanks{The work of N.~Bazhenov is supported by the Mathematical Center in Akademgorodok under the agreement No.~075-15-2025-349 with the Ministry of Science and Higher Education of the Russian Federation.
  The work of I.~Georgiev and S.~Vatev is supported by NextGenerationEU, through the NRRP of Bulgaria, project No.~BG-RRP-2.004-0008 - C01-70/123/195.
  The work of D.~Kalociński and M.~Wrocławski is supported by the National Science Centre Poland under the agreement No.~2023/49/B/HS1/03930.}

\title{Punctually Standard and Nonstandard Models of Natural Numbers}

\begin{document}

\maketitle

\begin{abstract}

Abstract models of computation often treat the successor function $S$ on
$\mathbb{N}$ as a primitive operation, even though its low-level implementations correspond to non-trivial programs operating on specific numerical representations. This behaviour can be analyzed without referring to notations by replacing the standard interpretation $(\mathbb{N}, S)$ with
an isomorphic copy $\mathcal A = (\mathbb{N}, S^{\mathcal A})$, in which
$S^{\mathcal A}$ is no longer computable by a single instruction. While the class of computable functions on \(\mathcal{A}\) is standard if \(S^{\mathcal{A}}\) is computable, existing results indicate that this invariance fails at the level of primitive recursion.
We investigate which sets of operations have the property that if they are primitive recursive on $\mathcal A$ then the class of primitive recursive functions on $\mathcal A$ remains standard. We call such sets of operations \emph{bases for punctual standardness}. We exhibit a series of non-basis results which show how the induced class of primitive recursive functions on $\mathcal A$ can deviate substantially from the standard one. In particular, we demonstrate that a wide range of natural operations—including large
subclasses of primitive recursive functions studied by Skolem and Levitz—fail
to form such bases.
On the positive side, we exhibit natural finite bases for punctual standardness. Our results answer a question recently posed by Grabmayr and establish punctual categoricity for certain natural finitely generated structures.

\end{abstract}

\section{Introduction}

When introducing a complexity class, one sometimes treats $(\mathbb{N}, S)$—that is, the natural numbers with the successor function—as a primitive notion. For example, loop-programs for primitive recursion \cite{MeyerR-67}, as well as while-programs (see, e.g., \cite{odifreddi_classical_1989}, Definition I.5.5) and register machines (see, e.g., \cite{shepherdson_computability_1963}) for partial recursive functions, all employ the instruction $X := X + 1$, which replaces the number stored in the register $X$ by its successor.

In contrast, symbolic models of computation such as Turing machines, Post machines or string rewriting systems do not take the successor operation as primitive but require it to be defined. For example, the successor operation on decimal numerals requires a non-trivial program that computes it. If one adopts a different notation, a different program for the successor is required. Consequently, the successor can be defined in many distinct ways, each time inducing a distinctive sequence of objects for representing the succession of numbers.

The same observation applies at a more abstract level, when considering the natural numbers: treating the successor operation as primitive can be replaced by a program whose denotation yields an isomorphic behavior. Even more directly, one may simply alter the semantics of \(X := X+1\). In either case, redefining the successor induces a new structure \(\mathcal A = (\mathbb N, S^{\mathcal A})\) isomorphic to \((\mathbb N, S)\). This affects our understanding of the successor (which now becomes \(S^{\mathcal A}\)) and, by extension, of the predecessor, addition, and all other functions.

This idea can be made precise as follows. Let \(c_\mathcal{A} : (\mathbb N, S) \to \mathcal A\) be an isomorphism. When the structure is clear from the context, we omit the subscript $\mathcal{A}$. The representation of a unary function \(f\) in \(\mathcal A\) is the isomorphic image of \(f\) relative to \(c\), namely \(c \circ f \circ c^{-1}\), succinctly denoted by \(f^{\mathcal A}\). This convention naturally extends to non-unary functions and relations. Crucially, along with this reinterpretation, the notion of a complexity class may also change. The following notion comes from computable structure theory: 
\begin{definition}\label{def computable on}
We say that a function $f$ is \emph{computable (resp. primitive recursive) on $\mathcal A$} if $f^{\mathcal A}$ is computable (resp. primitive recursive).
\end{definition}
Let $\mathcal C^{\mathcal A}$ denote the class of functions computable on $\mathcal A$. It is well known—and not difficult to show—that the equality $\mathcal C = \mathcal C^{\mathcal A}$ holds if and only if $S^{\mathcal A}$ is computable (see, e.g., \cite{shapiro_acceptable_1982}).

The equality $\mathcal C = \mathcal{C}^{\mathcal A}$ is desirable for the following reason: altering the semantics of the instruction $X := X+1$ should not change the class of computable functions, as long as we want to remain consistent with the Church--Turing thesis.
The condition $\mathcal C = \mathcal C^{\mathcal A}$ thus isolates the copies of $(\mathbb N, S)$ that we may call \emph{standard} or \emph{acceptable}. 
This notion has been investigated from several perspectives, including computability theory (acceptable numberings in the sense of Rogers~\cite{rogers_theory_1967}), theoretical computer science (acceptable programming systems in the sense of Riccardi~\cite{Riccardi} and Royer~\cite{Royer}), computable structure theory~\cite{Bazhenov-Kalocinski-23}, and philosophy~\cite{shapiro_acceptable_1982,shapiro_computability_2022}.


In this paper, we pursue this line of inquiry from the perspective of primitive recursion, using the methods of punctual structure theory \cite{kalimullin_algebraic_2017,bazhenov_foundations_2019,downey_foundations_2021,askes_online_2022}. Although the class of primitive recursive functions is very broad and extends well beyond practical computability, it nevertheless marks a fundamental theoretical boundary: it separates general computability from computability by algorithms that use only bounded loops \cite{MeyerR-67}. This boundary is familiar to computer science. There are natural and well-studied problems whose decidability has been known for decades, yet whose computational complexity was eventually shown to lie beyond primitive recursion. A prominent example is the reachability problem for Vector Addition Systems \cite{Reachability-VAS}.

 The central motivating observation of the present paper appears in the founding paper of punctual structure theory. While it is clear that, for a primitive recursive copy $\mathcal  A$ of $(\mathbb N,S)$, the isomorphism $c_\mathcal{A}$ (from the standard copy to $\mathcal A$) is primitive recursive, one can show the following:
\begin{proposition}[\cite{bazhenov_foundations_2019}, Example 4.1(3), p. 86]\label{motivating thm}
    There is $\mathcal{A}=(\mathbb{N},S^{\mathcal{A}})$, an isomorphic copy of $\mathcal{S}=(\mathbb{N},S)$, such that $S^{\mathcal{A}}$ is primitive recursive, but $c_\mathcal{A}^{-1}$ is not.
\end{proposition}
The structure $\mathcal A$ from Proposition \ref{motivating thm} is acceptable, that is, $\mathcal C = \mathcal C^{\mathcal A}$, since $S^{\mathcal A}$ is computable. However, by Observation \ref{observation} (proved below), the above proposition has the following corollary: there exists a primitive recursive function $f$ such that its representation $f^{\mathcal A}$ is not primitive recursive. If we denote by $\mathcal{PR}$ and $\mathcal{PR}^{\mathcal A}$ the classes of primitive recursive functions and of functions primitive recursive on $\mathcal A$, respectively, it follows that
$\mathcal{PR} \neq \mathcal{PR}^{\mathcal A}$. In other words, the class of functions that are primitive recursive on $\mathcal A$ does not coincide with the standard class of primitive recursive functions.

From the perspective of primitive recursive computation, the above situation is disquieting. In particular, if we alter in loop-programs the semantics of the instruction $X := X + 1$ by replacing it with $S^{\mathcal A}$ we obtain a nonstandard class of primitive recursive functions $\PR^\mathcal{A}$, despite $S^{\mathcal A}$ itself being primitive recursive\footnote{In fact, in this particular case one can compute $S^\mathcal{A}$ in linear time.} and despite the usage of apparently bounded loops. This runs counter to the restricted version of the Church--Turing thesis, an implicit quasi-formal assumption, formulated explicitly in Section~1.6 of~\cite{bazhenov_foundations_2019}: a function is primitive recursive if and only if it can be described by an algorithm that uses only bounded loops.


Hence, we are justified in introducing the following notion. Following \cite{kalimullin_algebraic_2017,bazhenov_foundations_2019}, by a \emph{punctual} structure we mean a structure with domain $\mathbb N$ whose signature relations and operations are primitive recursive.
\begin{definition}[punctual standardness]
We say that a punctual copy $\mathcal A$ of $\mathcal S = (\mathbb N, S)$ is \emph{punctually standard} if $\PR = \PR^\mathcal{A}$. Otherwise we say that $\mathcal A$ is \emph{punctually nonstandard}.
\end{definition}

A necessary and sufficient condition for punctual standardness can be formulated in terms of a suitably constrained isomorphism with the standard copy $\Std$. 
\begin{definition}[\cite{kalimullin_algebraic_2017}]
Punctual structures $\mathcal A$ and $\mathcal B$ are \emph{punctually isomorphic} if there exists a primitive recursive isomorphism $f$ between them whose inverse $f^{-1}$ is also primitive recursive. We call such an isomorphism $f$ \emph{punctual}.
\end{definition}
\begin{observation}\label{observation}
A copy $\mathcal A$ of $\Std$ with domain $\N$ is punctually isomorphic to $\Std$ if and only if $\mathcal A$ is punctually standard.

\end{observation}
\begin{proof}
If $\A$ is punctually isomorphic to $\Std$, then $c_\A$ and $c_\A^{-1}$ are primitive recursive bijections, therefore $\PR^\A = \{ c_\A \circ f \circ c_\A^{-1}\;|\; f \in \PR\} \;=\; \PR$ and $\A$ is punctually standard. Conversely, let $\A$ be punctually standard. For any $x$, we consider the numbers $c_\A(0),c_\A(1),\ldots,c_\A(x)$. Since $c_\A$ is injective, at least one of them is $\geq x$. We define $f(x) = \mu_{t \leq x} [c_\A(t) \geq x]$. Clearly, the function $f$ is primitive recursive and $c_\A(f(x)) \geq x$ for all $x$. But this also implies that $c_\A (f (c_\A^{-1}(x))) \geq c_\A^{-1}(x)$ for all $x$. Since $\A$ is punctually standard, the function $c_\A \circ f \circ c_\A^{-1}$ is primitive recursive. We conclude that $c_\A^{-1}$ is also primitive recursive and $\A$ is punctually isomorphic to $\Std$.
\end{proof}

A characterization of punctual standardness in terms of punctual isomorphism with the standard copy, while formally correct, has a drawback: it assesses punctual standardness relative to a particular representation of the natural numbers, namely $\Std$. In a sense, $\Std$ is punctually standard by definition, a choice that may be regarded as arbitrary. For this reason, it is reasonable to view this approach as external. Instead, we propose to seek a more general and robust \emph{internal} characterization. 

Guidance for such a characterization already comes from existing approaches to acceptability. For example, Rogers viewed acceptability of numberings of partial recursive functions as relative to the canonical numbering \cite{rogers_theory_1967}, but this approach was later replaced by characterizing acceptability internally via control structures \cite{Riccardi} or via realizability of natural arithmetic operations \cite{shapiro_acceptable_1982,Bazhenov-Kalocinski-23}. Following the second approach, we aim to isolate a preferably minimal and natural collection of basic arithmetic operations whose primitive recursiveness on $\mathcal A$ amounts to the punctual standardness of $\mathcal A$. A similar direction has been recently considered by Grabmayr (\cite{grabmayr_structure_2025}, p. 10).





\subsection{Contributions}\label{sec constributions}


The main new working concept of this paper is as follows.

\begin{definition}[basis for punctual standardness]\label{def:basis}
A class of functions $\mathcal F \subseteq \PR$ is a \emph{basis for punctual standardness} if for every copy $\mathcal A$ of $\mathcal S$ the following holds: if members of $\mathcal F$ are primitive recursive on $\mathcal A$, then $\mathcal A$ is punctually standard.

\end{definition}
Notice that the above concept can be generalized to an arbitrary complexity class of the functions on the natural numbers. In particular, by the results of Shapiro \cite{shapiro_acceptable_1982}, the class $\{S\}$ is a basis for acceptability, because $S$ is computable and for every copy $\mathcal A$ of $\Std$, if $S$ is computable on $\mathcal A$ then $\mathcal C = \mathcal C^\mathcal{A}$.

Our contribution is twofold. 

First, we exhibit several non-basis results---examples of natural classes that fail to provide bases for punctual standardness. Among those, we witness two phenomena. On the one hand, some punctual copies of $\Std$ seem to be very different from $\Std$ in the sense that surprisingly large natural subclasses of $\PR$ fail to remain primitive recursive in those copies (Contribution 1a); these examples are obtained by constructions that leverage the properties of the Ackermann function \cite{Ackermann}. On the other hand, we show that punctual copies of $\Std$ may remain punctually nonstandard despite preserving primitive recursiveness of large natural subclasses of $\PR$ (Contribution 1b); these examples are obtained by fine-tuning an island-technique, known from punctual structure theory, and then applying it to large classes of primitive recursive functions previously studied by Skolem and Levitz. Altogether, Contributions 1a and 1b show that finding bases for punctual standardness is quite non-trivial.

Second, we provide \emph{natural} bases for punctual standardness (Contribution~2). Their naturalness stems from the use of familiar arithmetic operations. The result is obtained through a complexity-theoretic analysis of the construction by Kalimullin et al.~\cite{kalimullin_algebraic_2017} and an application of known substitution bases for elementary functions. This resolves a technical question recently posed by Grabmayr~\cite{grabmayr_structure_2025}, which we discuss in more detail later. Moreover, our bases for punctual standardness immediately give rise to natural finitely generated \emph{punctually categorical} structures, that is, structures whose punctual copies are always punctually isomorphic. Previously, punctually categorical structures were known either through unnatural examples or in the non-finitely generated setting.

We describe each of the contributions in more detail below.

\subsubsection*{Contribution 1a}

The first set of negative results exhibits punctual copies of $\mathcal S$ on which certain natural operations are primitive recursive, yet nevertheless large classes of primitive recursive functions fail to remain primitive recursive on them. The proofs rely on a novel method that leverages properties of the Ackermann function—a paradigmatic example of a computable but non-primitive recursive function. This approach enables a systematic study of the resulting models; in particular, it allows us to show that certain broad classes of functions or relations are not primitive recursive on a given copy.

The following theorem is proved in Section \ref{sec contribution 1a technical} (see Proposition \ref{prop2x} and Lemma \ref{ineq}). The initial goal of the construction was to produce a punctual copy $\mathcal A$ of $\mathcal S$ such that the predecessor function is not primitive recursive on $\mathcal{A}$. In Proposition \ref{prop2x} we show that, in fact, many more natural primitive recursive functions are not primitive recursive on~$\mathcal A$.

\begin{theorem} There exists a punctual copy $\mathcal{A}$ of $\mathcal{S}$ such that the ordering is primitive recursive on $\A$ and whenever $f: \mathbb{N} \to \mathbb{N}$ is a primitive recursive function larger than $2x$ for all sufficiently large $x$, $f$ is not primitive recursive on $\mathcal{A}$.
\end{theorem}

The next theorem, proved in Section \ref{cont1a2} (see Proposition \ref{main1}), demonstrates a similar phenomenon: we can make addition primitive recursive, but still a great deal of natural primitive recursive functions may fail to be primitive recursive on a copy. 

\begin{theorem} There exists a punctual copy $\mathcal{B}$ of $\mathcal{S}$ such that the ordering and the addition are primitive recursive on $\mathcal{B}$, but whenever $f: \mathbb{N} \to \mathbb{N}$ is a primitive recursive function larger than $x^2$ for all sufficiently large $x$, $f$ is not primitive recursive on $\mathcal{B}$.
\end{theorem}

With the same technique, we can push this behavior up to addition and multiplication (see Proposition \ref{main2}):

\begin{theorem}\label{thm addition multiplication} There exists a punctual copy  $\mathcal{C}$ of $\mathcal{S}$ such that the ordering, the addition and the multiplication are primitive recursive on $\mathcal{C}$, but whenever $f: \mathbb{N} \to \mathbb{N}$ is primitive recursive and larger than $x^{\log_2 x}$ for all sufficiently large $x$, $f$ is not primitive recursive on $\mathcal{C}$.
\end{theorem}
The structures $\mathcal{B}$ and $\mathcal{C}$ are obtained from the structure $\mathcal{A}$ by introducing a general method, which manipulates base-$2$ expansions, to produce new punctual copies of $\mathcal{S}$ from already constructed ones. We show that it is not possible to extend this succession of theorems to the level of the exponential function by using the same method (see Proposition \ref{main3}). In contrast, the mainland--island technique is successful for preserving all exponential functions with a constant base (see Contribution 1b).

Theorem \ref{thm addition multiplication} has the following fundamental consequence, which we find surprising, given that  \(S, +, \times,\) and \(<\) form the core primitives in canonical axiomatizations of the natural numbers and constitute the most basic operations of ordinary arithmetic.

\begin{restatable}{corollary}{corArithmetic}\label{corollary_arithmetic}
The standard arithmetic operations of successor \(S\), addition \(+\), multiplication \(\times\), and ordering \(\leq\) do not constitute a basis for punctual standardness.
\end{restatable}
In this paper, we shall see three proofs of this result. Apart from Theorem \ref{thm addition multiplication}, it can also be derived by different means using Theorem~1.3 from \cite{KMZ-2024} and \cref{observation}. Yet another proof will be extracted from Theorem~\ref{theo:nonstandardness-meta}.






All the technicalities, including proofs, can be found in Section \ref{sec contribution 1a technical}.

\subsubsection*{Contribution 1b}

The novel techniques of Contribution~1a are very useful for a fine analysis of some particular (fast growing) primitive recursive functions. Nevertheless, in order to apply our methodology to familiar large classes of functions, we use the already established mainland--island method~\cite{kalimullin_algebraic_2017,bazhenov2020online}. Originally introduced in~\cite{kalimullin_algebraic_2017}, this method of construction is used to build a punctual isomorphic copy $\mathcal{A}$ of a given computable structure $\mathcal{M}$. The construction guarantees that primitive recursiveness of some functions $f^{\mathcal{A}}$ (typically successor, potentially also others) is preserved. Meanwhile, another function (either between elements of the domain of $\mathcal{A}$, or an isomorphism from $\mathcal{A}$ onto $\mathcal{M}$) is ensured to be not primitive recursive via diagonalisation.

In this paper, the island technique is showcased by the proof of the following result, which we were not able to prove using the method employed to obtain theorems described in Contribution~1a, except when $P$ is a singleton:
\begin{restatable}{theorem}{thmPrimes}
    \label{ax+b}
    For any finite set of prime numbers $P$, there exists $\mathcal{B}$---a punctual copy of $\mathcal{S}$---such that the only linear functions primitive recursive on $\mathcal{B}$ are $f(x)=ax+b$, where $a$ has no prime divisors outside of $P$.
\end{restatable}


To broaden the applicability of the mainland--island technique, we prove Theorem~\ref{theo:nonstandardness-meta} (see below), which provides sufficient conditions for constructing a punctually nonstandard copy $\mathcal{A}$ of the structure $\mathcal{S} = (\mathbb{N},S)$ via this method. Intuitively, one can view Theorem~\ref{theo:nonstandardness-meta} as a metatheorem: Theorem~\ref{theo:nonstandardness-meta} isolates technical conditions that preserve primitive recursiveness of a large class of functions $\mathcal{F}$ in $\mathcal{A}$, while ensuring that the image of the predecessor function $\Pred^{\mathcal{A}}$ is not classically primitive recursive. These conditions essentially combine the previous applications of the mainland--island method (see, e.g., \cite{kalimullin_algebraic_2017,bazhenov2020online,KMZ-2024,dorzhieva-hammatt}) and the recursion-theoretic properties of the class of functions studied by Levitz~\cite{Levitz}. 

In Theorem~\ref{theo:nonstandardness-meta}, $\mathcal{Q}$ is a uniform family of primitive recursive functions, and $\mathcal{F}$ is a uniform family of unary primitive recursive functions. All other technical conditions are fully explained in Section \ref{sec contribution 1b technical}.

\begin{restatable}{theorem}{theoMeta}\label{theo:nonstandardness-meta}
    If an $L$-family $(\mathcal{Q},\mathcal{F})$ has a set of normal forms, then the family $\mathcal{F}$ is not a basis for punctual standardness.
\end{restatable}

Theorem~\ref{theo:nonstandardness-meta} yields two important applications. First, it provides an independent justification of Corollary \ref{corollary_arithmetic}. Second, it can be used to derive a much stronger result.
Building on the work of Skolem~\cite{Skolem}, Levitz~\cite{Levitz} considered the following subclass of primitive recursive functions, that contains, in particular, all exponential polynomials.

\begin{restatable}{definition}{defLevitz}\label{def:Levitz}
The \emph{Levitz's class} $\LL$ is the least class of functions $f \colon \Nat \to \Nat$ containing the constant functions $0, 1$ and the identity $x \mapsto x $ and satisfying the following: if $f, g \in \LL$ and $n \in \Nat^+$, then the functions 
    \begin{enumerate}[label=(\alph*)]
        \item $x \mapsto f(x)+g(x)$, 
        \item $x \mapsto f(x)\cdot g(x)$, 
        \item $x \mapsto n^{f(x)}$, and 
        \item $x \mapsto x^{f(x)}$
    \end{enumerate}
    also belong to $\LL$.
    
    If we omit item~(d) from the definition of the Levitz's class, then the resulting class of functions will be denoted by $\LL_c$. Observe that $\LL_c \subset \LL$, and the class $\LL_c$ is closed with respect to composition.
\end{restatable}
\begin{restatable}{corollary}{corLevitz}\label{corollary_Levitz}
    The class $\LL_{c}$ is not a basis for punctual standardness.
\end{restatable}

Altogether, the results discussed so far show that finding a basis for punctual standardness is not straightforward. The most natural candidates---the arithmetic operations featured in the classical axiomatic descriptions of the natural numbers---fail to provide such a basis. Similarly, extending the potential basis to some natural fast growing functions comprising $\LL_c$ does not solve the problem either.

While we do not prove it in this paper, we remark that it is possible to directly construct a punctually nonstandard copy of $(\mathbb N, S, \Pred, +, \times, <)$. This alleviates the potential concern that punctual nonstandardness could arise from attempts to preserve only some non-decreasing functions. Furthermore, it seems plausible that $\LL_c \cup \{\Pred\}$ is not a basis for punctual standardness either.


All the technicalities, including proofs, can be found in Section \ref{sec contribution 1b technical}.

\subsubsection*{Contribution 2}

To obtain natural bases for punctual standardness, we begin with a complexity-theoretic analysis of the construction by Kalimullin, Melnikov, and Ng (\cite{kalimullin_algebraic_2017}, Proposition~4.2), which produces a rigid (that is, having only one automorphism) punctually categorical structure $\mathcal B = (\mathbb N, s,c,o)$.
 The structure $\mathcal B$ employs an artificial vocabulary with equally artificial interpretations, designed specifically to satisfy the intended requirements. For details, we refer the reader to the original paper; here, we only mention that $o$ is a constant, that $o, s(0), ss(o), \dots$ is an $\omega$-chain, and that $c$ is a unary function partitioning the domain into finite, pairwise disjoint cycles. Our analysis establishes that $\mathcal B$ is elementary in the sense of Kalm\'ar \cite{kalmar} (equivalently, $\mathcal E^3$ in the Grzegorczyk hierarchy). It follows that $(\mathbb N, S, s,c,o)$ is an elementary, rigid punctually categorical structure.
 
 The next step is to leverage existing results concerning substitution bases for the class of elementary functions. A substitution basis for a class $\mathcal F$ of functions (in this case, the elementary functions) is a finite collection of functions from $\mathcal F$ that generates the entire class through applications of the substitution operator together with projections and constants. The following is an example of such a basis: $x + y$, $x \!\mod y$, $x^2$, $2^x$ (see, e.g., Corollary 4.8 in \cite{mazzanti_plain_2002}). We prove that 
 
 \begin{restatable}{theorem}{elementarybasis}
 \label{thm:elementary}
     Every substitution basis for the class of elementary functions is a basis for punctual standardness.
 \end{restatable}

By Observation \ref{observation}, this result provides \emph{natural} examples of finitely generated punctually categorical structures. Previously such structures have been known either through the artificial finitely generated structure $\mathcal B = (\mathbb N, s,c,o)$ discussed above or through structures which are not finitely generated, such as the complete graph or the abelian \(p\)-group \(\bigoplus_{i\in\omega}\mathbb{Z}_p\).

Our result also provides, in a manner of speaking, a syntactic characterization of punctual standardness. Indeed, we describe this notion in a purely intrinsic manner by setting requirements on what functions are assumed to be primitive recursive, and without invoking any similarity criterion with a fixed standard copy.

Theorem \ref{thm:elementary} answers the question raised by Grabmayr (see, \cite{grabmayr_structure_2025}, p. 10) who requested to isolate complexity constraints on punctual number systems, sufficient to determine a unique punctual-isomorphism type. In the context of Grabmayr's paper, complexity constraints are understood as follows: which operations apart from the successor need to be primitive recursive in a number system (that is, in a punctual copy of $\Std$, in our terminology) to make the system punctually categorical? Here, additional complexity constraints are formalized by the condition ``members of $\mathcal F$ are primitive recursive on $\mathcal A$'' (see, Definition~\ref{def:basis}). It is worth noting that, given a substitution basis for $\mathcal E^3$,
 Theorem \ref{thm:elementary} not only isolates a unique punctual isomorphism-type---it isolates the \emph{right} punctual isomorphism type, that is: all punctual copies of such systems are punctually standard.

The technical part of Contribution 2 can be found in Section \ref{sec contribution 2 technical}.

\subsection*{Conclusions and Open Questions}
In this paper we addressed the following problem, recently (and independently)
raised by Grabmayr~\cite{grabmayr_structure_2025}: which operations must be
required to be primitive recursive on a copy $\mathcal A$ of $(\mathbb N,S)$ in
order to guarantee that the functions primitive recursive on $\mathcal A$
coincide exactly with the standard primitive recursive functions? By Observation \ref{observation}, this is equivalent to asking for a characterization of the punctual isomorphism type of $(\mathbb N,S)$. We have called
such sets of operations \emph{bases for punctual standardness} (Definition
\ref{def:basis}).

We established several negative results. In particular, the standard arithmetic operations $S,+,\times,<$, which feature prominently in canonical axiomatic characterizations of the natural numbers, do not constitute a basis for punctual standardness. At this point, the task of identifying such bases becomes non-trivial because few functions are at face value more natural than $S,+,\times,<$. Moreover, the range of non-basis results presented in this article shows that certain relatively natural and well-studied classes of primitive recursive functions fail to form such bases. To address the problem, an indirect strategy proved succesful: we took a categoricity result from punctual structure theory, analyzed the complexity of the underlying construction, and subsequently made use of existing results on substitution bases for the class of elementary functions.

Theorem \ref{thm:elementary} provides a natural structural characterization of
punctual standardness. For example, using one of the bases introduced by Mazzanti
\cite{mazzanti_plain_2002}, a punctual copy $\mathcal A$ of $(\mathbb N,S)$ is
punctually standard if and only if the functions $x+y$, $x \bmod y$, $x^2$, and
$2^x$ are primitive recursive on $\mathcal A$. Our result can be also framed as follows. Let $\mathcal F$ be a substitution basis for the class of elementary functions. Then the following equivalence holds: for every isomorphic presentation $\mathcal A = (\mathbb N, \mathcal F^\mathcal{A})$ of $(\mathbb N, \mathcal F)$, $\mathcal A$ is primitive recursive if and only if $\mathcal A$ is punctually standard. This shows that
the notion of punctual standardness is both natural and conceptually elegant.

Theorem \ref{thm:elementary} also has a fundamental further consequence: if all elementary functions are primitive recursive on $\mathcal A$, then the structure $\mathcal A$ remains punctually standard. Equivalently, any failure of punctual standardness must already be witnessed at the level of elementary functions. More strongly, at least one of the functions $x+y$, $x \bmod y$, $x^2$, or $2^x$ must fail to be primitive recursive on $\mathcal A$. This reinforces the view that punctually nonstandard copies of the natural numbers with successor are highly unnatural objects.

Finally, it is instructive to juxtapose this picture with our negative results, in particular Corollary \ref{corollary_Levitz}. This corollary shows that a large and natural class of functions, $\mathcal L_c$, is insufficient to guarantee punctual standardness. Although $\mathcal L_c$ is contained within the class of elementary functions, it omits certain key elementary operations whose absence allows other elementary functions to become prohibitively complex. This highlights, once again, the delicate balance required for a set of operations to serve as a basis for punctual standardness.

Several open questions arise.


While negative results have a limit, in many cases it is not clear where this limit lies. In particular, we do not know how far Theorem~\ref{theo:nonstandardness-meta} can be pushed further and whether the whole Levitz's class constitutes a basis for punctual standardness. We conjecture that it does not. Other large classes of primitive recursive functions are natural directions for further investigation. Examples are those considered by Barra and Gerhardy in \cite{Barra-Gerhardy}, defined using iterated exponentiation.  

We also conjecture that already the second Grzegorczyk class $\mathcal{E}^2$ is a basis for punctual standardness. This would require a substantial modification of the construction in \cite{kalimullin_algebraic_2017} of a~punctually categorical rigid structure.



\section{Preliminaries}\label{sec prelim}

We refer to the monograph~\cite{odifreddi_volume_2} for the detailed background on primitive recursion and the Grzegorczyk hierarchy. The reader may also consult a textbook on subrecursive hierarchies~\cite{Rose1984-ROSSFA-2}.

Let $f$ and $g$ be total functions.
We denote $f \leq_{PR} g$ ($f \leq_E g$) if $f$ can be obtained from $g$ and the initial elementary functions using composition and (bounded) primitive recursion.

As usual,
$f \equiv_E g\ \Leftrightarrow\ f \leq_E g \;\&\; g \leq_E f$ and $f \equiv_{PR} g \ \Leftrightarrow\ f \leq_{PR} g \;\&\; g \leq_{PR} f$.

A function $f$ is primitive recursive (elementary), if $f \leq_{PR} 0$ ($f \leq_E 0$).

A relation $R$ is primitive recursive (elementary), if its characteristic function is primitive recursive (elementary).

We assume a fixed (surjective) coding of pairs: $\pair{x,y}$ is the code of the pair $(x,y)$,
where we denote $x = (\pair{x,y})_0$ and $y = (\pair{x,y})_1$. As usual, the functions $(x,y) \mapsto \pair{x,y}, \; z \mapsto (z)_0, \; z \mapsto (z)_1$ are chosen elementary. Moreover, $\pair{x,y}$ is strictly increasing in both $x, y$, which implies that $(z)_0 < z$ and $(z)_1 < z$ for all non-zero $z$.

Let $\mathcal{T}_1$ be the elementary Kleene's predicate, so that any partial computable unary function has the form $\lambda x.(\mu s\mathcal{T}_1(e,x,s))_0$. We can use the Ackermann function $A(x,y)$ to define a computable enumeration of the class of all unary primitive recursive functions:
  $$ \psi_e(x) \;=\; (\mu s\leq A((e)_1,x)\; \mathcal{T}_1((e)_0,x,s))_0. $$
 We also define 
  $$ \psi_{e,s}(x)\downarrow \;\;\Longleftrightarrow\;\; \mathcal{T}_1((e)_0,x,s), $$
 which signifies that the computation of $\psi_e(x)$ takes $\leq s$ steps. We will frequently use $\psi_{e,s}(x)\!\uparrow$ for the negation of $\psi_{e,s}(x)\!\downarrow$. We may assume the following property: 
        \begin{gather}\label{equ:max}
         \psi_{e,s}(x)\!\downarrow\ \ \Longrightarrow\ s> \max(e,x,\psi_e(x)).
        \end{gather}

Let $a : \Nat \to \Nat$ be a sequence with the following properties:
\begin{enumerate}[label=\arabic*.]
  \item $a$ is strictly increasing; \label{p1}
  \item the relation $a_n = k$ is elementary; \label{p2}
  \item for any unary primitive recursive function $f$, $f(a_n) < a_{n+1}$ for all sufficiently large $n$. \label{p3}
\end{enumerate}
It is well-known that $a_n = A(n,n)$ satisfies these three properties.

The following observations are immediate:
\begin{enumerate}[label=(\roman*)]
 \item Property \ref{p3} implies that for any unary primitive recursive function $f$, $f(n) < a_n$ for all sufficiently large $n$, thus $a$ is not primitive recursive.

 \textit{Proof.} Take $g(n) = \displaystyle{\max_{k \leq n+1} f(k)}$. Then $g$ is primitive recursive and $$f(n+1) \leq g(n) \leq g(a_n) < a_{n+1}$$ for all but finitely many $n$.

 \item Properties \ref{p1} and \ref{p2} imply that there exists an elementary function
 \begin{equation}\label{leminv}
   a^{-1} : \Nat \to \Nat\;\;\;\;\text{with}\;\;\;\;a^{-1}(a_n) = n. \tag{$*$}
 \end{equation}
\textit{Proof.} Given input $k$, search for $n \leq k$, such that $a_n = k$. If successful, give output $a^{-1}(k) = n$. Otherwise, give output $0$.
 \item As a unary relation $Ran(a)$ is elementary: $$k \in Ran(a) \;\Longleftrightarrow\; \exists n\leq k\; a_n = k.$$
 \item The unique strictly increasing unary function $h$ with $ran(h) = \Nat\setminus Ran(a)$ is elementary.
 \textit{Proof.} Property \ref{p3} implies that for all sufficiently large $x$, $\{x+1,x+2\} \cap (\Nat\setminus Ran(a)) \neq \emptyset$ ($Ran(a)$ is a very sparse set). Therefore $h(x) \leq 3x$ for all but finitely many $x$.
\end{enumerate}

\section{Examples for Punctually Nonstandard Structures}\label{sec contribution 1a technical}
  We provide a new idea to build a punctual copy $\A = (\Nat, S^\A)$ of $(\Nat, S)$, such that the predecessor function is not primitive recursive on $\A$. Unlike the constructions in \cite{kalimullin_algebraic_2017} and \cite{bazhenov_foundations_2019}, this idea allows for a much finer analysis of the primitive recursive functions and relations on the constructed copy $\A$. 
  
  We start the definition of $\A$ by $S^\A(a_n) = h(\pair{n,0})$. Since $a$ is not primitive recursive, we cannot compute $a_n$ from $h(\pair{n,0})$ by a primitive recursive algorithm, but (\ref{leminv}) implies that given $a_n$ we can elementarily compute $h(\pair{n,0}) = h(\pair{a^{-1}(a_n),0})$. Now we take care to connect all these elements in an elementary manner. Here is an illustration of the idea:
  $$ a_n \longrightarrow h(\pair{n,0}) \longrightarrow h(\pair{n,1}) \longrightarrow \ldots \longrightarrow h(\pair{n,a_{n+1}}) \longrightarrow a_{n+1}. $$
  The natural numbers that remain unused are those in the set $$F = \{h(\pair{n,i})\;|\; i > a_{n+1}\}.$$
  Clearly, $F$ is elementary, because $h$ and the graph of $a$ are elementary. The unique strictly increasing unary function $free$ with $Ran(free) = F$ is primitive recursive: if we have computed $h(\pair{n,i}) \in F$, then the value of the next element in $F$ is at most $h(\pair{n,i+1})$.

  To complete our model, we must insert the elements of $F$ in the above picture (every natural number must appear exactly once).
  We can do that in the following way: 
   $$ a_n \longrightarrow free(n) \longrightarrow h(\pair{n,0}).$$

  More formally, we define $S^\A$ by the following primitive recursive algorithm.
  
  Given input $x$:

  1. If there exists $n \leq x$, such that $a_n = x$, $S^\A(x) = free(n)$.
  
  2. If there exists $n \leq x$, such that $free(n) = x$, $S^\A(x) = h(\pair{n,0})$.

  3. Otherwise, there exists $k \leq x$, such that $h(k) = x$. We compute $n, i$ with $k = \pair{n,i}$.
  If $i = a_{n+1}$, then $S^\A(x) = i$, else $S^\A(x) = h(\pair{n,i+1})$.

  Note that in Clause 3, we have $i \leq a_{n+1}$, because the case $i > a_{n+1}$ is covered by Clause 2.\newline

  Let $c$ be the unique isomorphism from $(\Nat,S)$ to $(\Nat,S^\A)$.
  
  For any $x\in\Nat$, let us call \emph{the $\A$-number of} $x$ the unique $p\in\Nat$, such that $x = c(p)$, so that the $\A$-number of $x$ is $c^{-1}(x)$.

  An easy induction on $n$ shows that in the above model $\A$:
  
  \qquad the $\A$-number of $a_n$ is $3n + \sum_{k=1}^n a_k$,
  
  \qquad the $\A$-number of $free(n)$ is $3n+1+\sum_{k=1}^n a_k$ and
  
  \qquad the $\A$-number of $h(\pair{n,i})$ is $3n + i + 2 + \sum_{k=1}^n a_k$ for $i \leq a_{n+1}$.

  On one hand, the isomorphism $c$ is primitive recursive, because $S^\A$ is primitive recursive and $c(p) = (S^\A)^p(a_0)$. On the other hand, $c^{-1}$ is not primitive recursive, because $c^{-1}(free(n)) > a_n$ for all $n$. In fact, it is easy to prove that $c^{-1} \equiv_E a$.

  Given $p$ we can primitive recursively compute the element $x = c(p)$ with $\A$-number $p$, but given $x$ we cannot primitive recursively compute the $\A$-number $p$ of $x$.

  Also note that for a function $f : \Nat \to \Nat$, the $\A$-number of $f^\A(x)$ is $f(p)$, where $p$ is the $\A$-number of $x$.

\subsection{Exploring Images of Functions in the Model $\A$}

  We explore the complexity of images of functions in the constructed model $\A$. For brevity, we will omit $\A$ and use ``number'' instead of ``$\A$-number'', and similarly $c$ instead of $c_\A$ for the notation for the isomorphism.
  
  Our first example is the image $\Pred^\A$ of the predecessor function, which is not primitive recursive. Otherwise, $\Pred^\A(free(n)) = a_n$ would be primitive recursive.

  It turns out that many other primitive recursive functions $f$ are not primitive recursive in the constructed model $\A$.

  \begin{proposition}\label{prop2x} Let $f : \Nat \to \Nat$ be a primitive recursive function, such that $f(x) \geq 2x$ for all sufficiently large $x$. Then $f^\A$ is not primitive recursive.
  \end{proposition}
  \begin{proof} Suppose $f^\A$ is primitive recursive. Then $f^\A(free(n))$ is primitive recursive in $n$.
    The number of $f^\A(free(n))$ is $f(p)$, where $p$ is the number of $free(n)$. So the number of $f^\A(free(n))$ is $f(3n+1+\sum_{k=1}^n a_k)$. Property \ref{p3} implies that for all sufficiently large $n$ we have
    $$ f\left(3n+1+\sum_{k=1}^n a_k\right) \;<\; a_{n+1}, $$
    so the number of $f^\A(free(n))$ is strictly between the numbers of $free(n)$ and $a_{n+1}$ and by construction $f^\A(free(n)) = h(\pair{n,i})$ for some $i \leq a_{n+1}$. But the number of $h(\pair{n,i})$ is $3n + i + 2 + \sum_{k=1}^n a_k$, therefore:
    $$ 3n + i + 2 + \sum_{k=1}^n a_k \;=\; f\left(3n+1+\sum_{k=1}^n a_k\right) \;\geq\; 2\left(3n+1+\sum_{k=1}^n a_k\right), $$
    which clearly implies the inequality $i \geq a_n$ for all sufficiently large $n$. But this is a contradiction, since $i$ can be computed primitive recursively from $n$ using the equality $f^\A(free(n)) = h(\pair{n,i})$.
  \end{proof}
  As a corollary, the only non-constant unary polynomials $p$, such that $p^\A$ is primitive recursive are those of the form $p(x) = x+b$ (for such $p$, $p^\A(x) = (S^\A)^b(x)$).\newline

  It seems that few primitive recursive functions $f$ remain primitive recursive on $\A$.
  But whenever the isomorphism $c$ is primitive recursive, the function $f \circ c$ remains primitive recursive on $\A$ for any primitive recursive $f$, because $(f \circ c)^\A = c \circ f \circ c \circ c^{-1} = c \circ f$.
  So in some sense, among the functions which remain primitive recursive on $\A$, an entire copy of all primitive recursive functions resides. But in another sense, these functions are quite unnatural (compared to polynomials, for example), because their definition crucially depends on the isomorpishm $c$.

  We have seen examples for primitive recursive $f$ in which $f^\A$ remains primitive recursive or becomes not primitive recursive. In the next proposition we will see a family of examples, in which $f$ is not primitive recursive, but $f^\A$ becomes primitive recursive.

  \begin{proposition}
  For any primitive recursive $v : \Nat \to \Nat$, such that $v(x) > x$, there exists $u : \Nat \to \Nat$
  such that $u \equiv_E v$, $u = f^\A$, and $f$ is not primitive recursive.
  \end{proposition}
  \begin{proof} The idea of the proof is that in the standard model it is easy to jump primitively recursively from an element of $F$ to the next element of $F$, but in the model $\A$ the elements of $F$ become scattered and these jumps are no longer possible.

    Given a primitive recursive $v$ with $v(x) > x$, let us define $u$ in the following way:
    $$ u(x) = \begin{cases}
                 h(\pair{m, v(i)}), & \text{ if } x = h(\pair{m,i}),\\
                 0, & \text{ otherwise}.
              \end{cases} $$
    Since $h$ is elementary, we have $u \leq_E v$ and also $v \leq_E u$, because $v(i)$ can be computed from $u(h(\pair{0,i})) = h(\pair{0,v(i)})$. Of course, $u$ is primitive recursive and let us take $f = c^{-1} \circ u \circ c$,
    so that $f^\A = u$. Assume that $f$ is primitive recursive. We will show a primitive recursive algorithm to compute $a_{s+1}$ from $a_s$, which is impossible by Property \ref{p3}.
    
    Given input $a_s$:
    
    1. Compute $y$, such that $c(y) = free(s)$.

    2. Compute $f(y)$ and give output $a_{s+1}$.

    We can compute $y$ in the Step 1, since $y = 3s+1+\sum_{k=1}^s a_k$ and through the graph of $a$, we have access to $a_k$ for all $k \leq s$. Let $free(s) = h(\pair{m,i})$. Since $v(i) > i$ and $h(\pair{m,i}) \in F$, we have $u(h(\pair{m,i})) = h(\pair{m, v(i)}) \in F$. So $u(c(y)) = u(free(s)) = free(t)$ for $t \geq s+1$ and we conclude that the number of $u(c(y))$ is at least as large as the number of $free(s+1)$, that is $f(y) = c^{-1}(u(c(y))) \geq 3s+4+\sum_{k=1}^{s+1} a_k$. In particular, $a_{s+1} \leq f(y)$ and we can compute $a_{s+1}$ in Step 2 using the value $f(y)$ and the graph of $a$.  
  \end{proof}

  It turns out there are similar examples with more natural choice of the function $u$.
  
  \begin{proposition}
   Let $g$ be the function, such that $g^\A$ is the predecessor function. Then $g$ is not primitive recursive.
  \end{proposition}
  \begin{proof}
    Assume $g$ is primitive recursive. As in the previous proposition, we will show that $a_s \mapsto a_{s+1}$ is primitive recursive, leading to a contradiction.

    We may assume that $s$ is sufficiently large, so that $free(s+1) < a_s$ and $h(\pair{s,a_s}) < a_{s+1}$.\newline
    
    Given input $a_s$:

    Try the values $x = free(s+1)+1,\; free(s+1)+2,\; \ldots,\; h(\pair{s, free(s+1)+1})$ successively:

    1. If $x \notin Ran(h)$, then terminate.

    2. If $x \in Ran(h)$, compute $m, i$, such that $x = h(\pair{m,i})$.

    \;\;\;\;\;2.1. If $x \notin F$ and $m \leq s$, then terminate.

    \;\;\;\;\;2.2. Else, $x = x+1$ and continue to 1.

    3. Compute the number $p$ of $x$ and then compute $a_{s+1}$ using $g(p)$.\newline

    First observe that the search for $x$ always terminates, because the last value $x = h(\pair{s, free(s+1)+1})$ satisfies the condition in 2.1.

    Consider the value of $x$ at the step of termination. If $x \notin Ran(h)$, then $x = a_k$ with $k \leq s$ (the above inequalities show that $k > s$ is not possible).
    If $x \in Ran(h)$, then $x = h(\pair{m,i})$ with $m \leq s$ and $i \leq a_{m+1}$. In both cases we can compute the number $p$ of $x$, because we have access to $a_0, a_1, \ldots, a_s$.
    Now consider the number $x-1$ from the previous step. Either $x-1 = free(s+1)$ (termination at the first step) or $x-1 \in F$ or $x-1 = h(\pair{m,i})$ with $m \geq s+1$ (termination at a later step).
    The value of $x$ only increases, therefore $x-1 = free(t)$ for $t \geq s+1$ or $x-1 = h(\pair{m,i})$ with $m \geq s+1$ and $i \leq a_{m+1}$. In both cases, the number of $x-1$ is clearly greater then $a_{s+1}$. By the definition of $g$, the number of $x-1$ is $g(p)$, therefore we can compute $a_{s+1}$ using $g(p)$ and the graph of $a$.  
  \end{proof}

\subsection{Exploring Images of Relations in the Model $\A$}

  Let $R$ be a unary relation. Its image $R^\A$ is defined by: $R(x) \Leftrightarrow R^\A(c(x))$ or equivalently, $R^\A(x) \Leftrightarrow R(c^{-1}(x))$.

  \begin{lemma}\label{ineq} The image $<^\A$ of the standard ordering is primitive recursive.
  \end{lemma}
  \begin{proof} We are given $x, y$ and we want to decide whether $c^{-1}(x) < c^{-1}(y)$ or $c^{-1}(x) > c^{-1}(y)$ (of course, $c^{-1}(x) = c^{-1}(y)$ if and only if $x = y$). We cannot use $c^{-1}$, but we can compute $m$, such that $x$ is in the portion of the model $\A$ between $a_m$ and $a_{m+1}$ (and the similar $n$ corresponding to $y$). To this end, we first describe the primitive recursion function $pos$:

    Given input $x$:

    1. If $\exists\; m \leq x$ with $a_m = x$, set $i = 0$.

    2. If $\exists\; m \leq x$ with $free(m) = x$, set $i = 1$.

    3. Otherwise, compute $m, i'$, such that $x = h(\pair{m,i'})$ and set $i = i'+2$.

    4. Return answer $pos(x) = \pair{m,i}$.\newline

   Now given $x, y$ we compute $pos(x) = \pair{m,i}$ and $pos(y) = \pair{n,j}$.
   
   If $m < n \;\vee\;(m = n \;\&\; i < j)$, then $x <^\A y$.
   
   If $m > n \;\vee\;(m = n \;\&\; i > j)$, then $y <^\A x$.
  \end{proof}

  Clearly, $R \leq_{PR} R^\A$. Therefore, if $R^\A$ is primitive recursive, then $R$ is also primitive recursive.
  Moreover, for any primitive recursive $S$, there exists a primitive recursive $R$ such that $R^\A = S$.
  Indeed, we can take $R(x) \Leftrightarrow S(c(x))$.

  So the image of a subclass of the primitive recursive relations in the model $\A$ is the class of all primitive recursive relations. We will see that this subclass is proper by exhibiting examples in which $R$ is primitive recursive, but $R^\A$ is not.

  By replacing $a_n$ with $\max(a_n,A(n,n))$, we may assume $A(n,n) \leq a_n$ for all $n$.
  
  This implies that we can choose a primitive recursive function $\phi$,
  such that $\psi_n(n) \;=\; \phi(a_n)$ for all $n$ (recall that $\psi$ is the enumeration of the unary primitive recursive functions, defined in Section \ref{sec prelim}).

  Consider now the relation $S$, defined by the following algorithm.
  
  Given input $x$:

  1. Compute the largest $a_k$ with $a_k \leq c^{-1}(x)$.
  
  2. Return output $1$ if $\phi(a_k) = 0$ and output $0$ if $\phi(a_k) \neq 0$.

  \begin{proposition} The relation $R$ defined by $R(x) \Leftrightarrow S(c(x))$ is primitive recursive,
    but the relation $R^\A = S$ is not primitive recursive.
  \end{proposition}
  \begin{proof}
    First we argue that $R$ is primitive recursive. Given input $c(x)$ to $S$, using the graph of $a$,
    we can determine the largest $a_k$ with $a_k \leq c^{-1}(c(x)) = x$ and then return the output using the primitive recursive function $\phi$.

    It is clear that $R^\A = S$, so it remains to show that $S$ is not primitive recursive.

    Given $n$, we can take the input $x = h(\pair{n,0})$. Then $c^{-1}(x)$ is the number of $x$,
    that is $c^{-1}(x) = 3n + 2 + \sum_{k=1}^{n} a_k$. The largest $a_k$ with $a_k \leq c^{-1}(x)$ is $a_n$ for sufficiently large $n$. Therefore, $\phi(a_k) = \phi(a_n) = \psi_n(n)$. We obtained that for almost all $n$,
    $S(h(\pair{n,0}))$ returns output, which is different from $\psi_n(n)$, therefore $S$ cannot be primitive recursive.
  \end{proof}

  We can obtain many other examples in the same spirit.
  
  For any primitive recursive function $f$, such that $f(x) \geq x$,
  let us take the relation $R_f$, defined by $R_f(x) \Leftrightarrow S(c(f(x)))$. Basically the same proof gives that $R_f$ is primitive recursive and $R_f^\A$ is not primitive recursive. Given input $x = h(\pair{n,0})$,
  $R_f^\A(x) \Leftrightarrow S(y)$, where $y = c(f(c^{-1}(x)))$. So the largest $a_k \leq c^{-1}(y) = f(3n + 2 + \sum_{k=1}^n a_k)$ is again $a_n$ for sufficiently large $n$.

\subsection{Enforcing Primitive Recursive Images of Linear Functions}

The natural next step is to produce a nonstandard punctual copy $\D$, in which the image of the doubling function is primitive recursive. We take the punctual copy $\A$, we replace its domain with the even numbers and put its $n$-th element at position $2^n$ in the new copy $\D$. In such a way, the image of the doubling function in $\D$, restricted to the elements of $\A$, coincides with the image of successor in $\mathcal{A}$. In addition, we must take care to fill all the other positions with the odd numbers in such a way that the image of doubling over them is also primitive recursive.

In more detail: We take the model $\A = (\Nat,S^\A)$ and replace $\Nat$ by $2\Nat$ so that the odd numbers become free. Let us take two elements $2x$ and $2y$, such that $S^\A(x) = y$ and assume we have inserted elements preceding $2x$, so that the new number of $2x$ is $p'$, where $p' \geq 2$. Then we insert the following odd numbers between $2x$ and $2y$ :
  $$ 2x \;\rightarrow\; odd'(x) \;\rightarrow\; 2\pair{x,0} + 1 \;\rightarrow\; 2\pair{x,1} + 1 \;\rightarrow\; \ldots \;\rightarrow\; 2\pair{x,k_x-1} + 1 \;\rightarrow\; 2y, $$
  so that the new number of $2y$ is $2p'$ and $odd'$ is the strictly increasing enumeration of the set $\{2\pair{x,i} + 1\;|\;i \geq k_x\}$. Let us call this new copy $\mathcal{D} = (\Nat,S^\mathcal{D})$. It begins in the following way:
  \begin{align*}
   &2c(0) \;\rightarrow\; odd'(c(0)) \;\rightarrow\; 2c(1) \;\rightarrow\; odd'(c(1)) \;\rightarrow\\
   &2c(2) \;\rightarrow\; odd'(c(2)) \;\rightarrow\; 2\pair{c(2),0} + 1 \;\rightarrow\; 2\pair{c(2),1} + 1 \;\rightarrow\; 2c(3) \;\rightarrow\; \ldots
  \end{align*}
  We have $k_{c(0)} = 0$ and for $p \geq 1$, $k_{c(p)} = 2^p - 2$. The new number $p'$ of $2x$ is $2^p$, where $p = c^{-1}(x)$ is the old number of $x$. It follows that the relation $k_x = i$ is primitive recursive: $k_x = i \Leftrightarrow \exists p \leq i+1\; [x = c(p) \;\&\; \max(2^p,2) = i+2]$.
  
  So the set $Ran(odd')$ and its enumeration $odd'$ are also primitive recursive.

  \begin{proposition} The successor $S^\mathcal{D}$ is primitive recursive.
  \end{proposition}
  \begin{proof} The algorithm is a straightforward formalization of the definition of the model $\mathcal{D}$. We will use that $\mathcal{S}^\A$ is primitive recursive and that the relation $k_x = i$ is primitive recursive, which implies that the relations $k_x < i$ and $k_x > i$ are also primitive recursive.

  Given input $n$:

  1. If there exists $t \leq n$, such that $2t+1 = n$:
  
\;\;\; Compute $x,i$, such that $t = \pair{x,i}$.

\;\;\;\;\;    1.1. If $k_x > 0 \;\;\&\;\; i+1 = k_x$, then $S^\mathcal{D}(n) = 2S^\A(x)$.

\;\;\;\;\;    1.2. If $k_x > 0 \;\;\&\;\; i+1 < k_x$, then $S^\mathcal{D}(n) = 2\pair{x, i+1} + 1$.

\;\;\;\;\;    1.3. If $k_x = 0$ or $(k_x > 0 \;\;\&\;\; i+1 > k_x)$, then:

\;\;\;\;\;\;\; Compute $m$, such that $odd'(m) = n$.

\;\;\;\;\;\;\;\;\; 1.3.1. If $k_m = 0$, then $S^\mathcal{D}(n) = 2S^\A(m)$.
                   
\;\;\;\;\;\;\;\;\; 1.3.2. Otherwise, $S^\mathcal{D}(n) = 2\pair{m,0} + 1$.
  
  2. Otherwise, $S^\mathcal{D}(n) = odd'(\lfloor \frac{n}{2} \rfloor).$
\end{proof}

\begin{proposition} The image $d^\mathcal{D}$ of the function $d(x) = 2x$ is primitive recursive.
\end{proposition}
\begin{proof} The idea is the following: if we have as input $2x$, then we can jump to $d^\mathcal{D}(2x) = 2y$ using the old successor $S^\A$. If we have as input $odd'(x)$ or $2\pair{x,i} + 1$, we can compute $x$, jump to $2y$ as in the first case and then go right twice or $2i+4$ times, accordingly. This crucially depends on the fact that the difference $d(p+1) - d(p)$ is constant.
  
  Given input $n$:

  1. If there exists $t \leq n$, such that $2t+1 = n$:
  
\;\;\; Compute $x,i$, such that $t = \pair{x,i}$.

\;\;\;\;\;    1.1. If $(k_x > 0 \;\;\&\;\; i+1 \leq k_x)$, then $d^\mathcal{D}(n) = 2\pair{2S^\A(x), 2i+2} + 1$.

\;\;\;\;\;    1.2. Otherwise:

\;\;\;\;\;\;\; Compute $m$, such that $odd'(m) = n$ and $y = S^\A(m)$.

\;\;\;\;\;\;\; If $k_y > 0$, then $d^\mathcal{D}(n) = 2\pair{y,0} + 1$, otherwise $d^\mathcal{D}(n) = 2y$.
  
  2. Otherwise:
  
  If $n = 2c(0)$, then $d^\mathcal{D}(n) = n$, else $d^\mathcal{D}(n) = 2S^\A(\lfloor \frac{n}{2} \rfloor).$
\end{proof}

\begin{proposition}
  The only linear functions whose images are primitive recursive in the model $\mathcal{D} = (\Nat,S^\mathcal{D})$ are those of the form $2^a x + b$.
\end{proposition}
\begin{proof} Clearly, the image of the function $\lambda x.2^a x + b$ is $\lambda x.((S^{\mathcal{D}})^b(d^{\mathcal{D}})^a(x))$, therefore it is primitive recursive.

  Now let us consider the linear function $f(x) = mx + b$, where $2^{a} < m < 2^{a+1}$ for some $a$. We will show that $c^{-1} \leq_{PR} f^\mathcal{D}$ and therefore, $f^\mathcal{D}$ cannot be primitive recursive.

  Given input $x$:

  1. If $c^{-1}(x) \leq b$, then give the answer using a table.

  2. Else compute $f^\mathcal{D}(2x)$ and then $x', i$, such that $f^\mathcal{D}(2x) = 2\pair{x',i} + 1$.
  
  3. Search for $p \leq i+1$, such that $c(p) = x$ and give answer $p$.\newline

  Let $p = c^{-1}(x)$ be the $\A$-number of $x$. In Step 2, $p > b$ and the $\D$-number of $2x$  is $p' = 2^p \geq 2$, therefore the $\D$-number of $f^\mathcal{D}(2x)$ is $mp' + b$. Now we have (these are the $\D$-numbers of the corresponding elements)
  $$ 2^a p' + 2 \;\leq\; (m-1)p' + 2 \;\leq\; mp' \;\leq\; mp'+b \;<\; mp'+p \;<\; mp'+p' \;\leq\; 2^{a+1}p'. $$
  By the construction of the model $\mathcal{D}$, we have $f^\mathcal{D}(2x) = 2\pair{x',i} + 1$,
  where $i = (mp'+b) - (2^a p'+2) \;\geq\; p' - 2$ (the $\D$-number of $2x'$ is $2^ap'$, but we do not need this fact).
  We obtained $p' \leq i+2$, that is $2^p \leq i+2$, which implies $p \leq i+1$ and the search in 3. is successful.
\end{proof}

\begin{proposition}\label{propsq} Let $f : \Nat \to \Nat$ be a primitive recursive function, such that $f(x) \geq x^2$ for all sufficiently large $x$.
  Then $f^\mathcal{D}$ is not primitive recursive.
\end{proposition}
\begin{proof}
  Let $f(x) \geq x^2$ for almost all $x$ and assume $f$ and $f^\mathcal{D}$ are primitive recursive.

  We will define a new function $g^\A$ and then we will argue that its preimage $g$ is primitive recursive and $g(p) \geq 2p$ for almost all $p$, which is impossible by Proposition \ref{prop2x}.

  In order to compute $g^\A(x)$: compute $f^\mathcal{D}(2x)$ and give answer $y$, where $2y$ is the smallest even element, whose $\D$-number is greater than the $\D$-number of $f^\mathcal{D}(2x)$.

  Observe that $f^\mathcal{D}(2x)$ has the form $2x'$, $odd(x')$ or $2\pair{x',i} + 1$ for some $x', i$, therefore $y = S^\A(x')$ and thus $g^\A$ is primitive recursive.

  If $x$ has sufficiently large $\A$-number $p$, then $2x$ has $\mathcal{D}$-number $2^p$, therefore $f^\mathcal{D}(2x)$ has $\mathcal{D}$-number $f(2^p) \geq 2^{2p}$, so the $\mathcal{D}$-number of $2y$ is also greater than $2^{2p}$. It follows that the $\A$-number of $y$ is greater than $2p$, that is $g(p) \geq 2p$.

  In addition, the $\A$-number $g(p)$ of $y$ can be computed primitively recursively from $f^\mathcal{D}(2x)$ and its $\mathcal{D}$-number $f(2^p)$.
  Since $f$ is primitive recursive, $g$ is also primitive recursive. So the proof is finished.
\end{proof}

\subsection{Enforcing Primitive Recursive Addition and Multiplication in the Model}\label{cont1a2}

The idea from the previous section cannot be used to obtain that the doubling and the tripling functions have primitive recursive images simultaneously. For example, in the model $\D$ the element at position $3 \cdot 2^n$ = $2^{n+1} + 2^n$ lies in the middle between the $n$-th and the $(n+1)$-st element of $\mathcal{A}$ and is therefore unaccessible without knowing $n$.

One approach to overcome this problem is to use an mainland-island construction. This is done in the proof of Theorem \ref{ax+b} in Section \ref{sec contribution 1b technical}.

Toward an alternative approach, we observed that a hierarchical structure can be imposed on the elements whose positions lie between $2^n$ and $2^{n+1}$ by using not single elements of the original copy $\mathcal{A}$, but lists of such elements. For example, the list $\langle c(n) \rangle$ with length $1$ obtains position $2^n$ for all $n$. In order to double the position, just apply the successor $S^\mathcal{A}$ to the elements of the list. In order to triple the position of $\langle c(n) \rangle$, we put the element $\langle c(n), c(n+1) \rangle$ at position $3 \cdot 2^n$. Thus we gradually came to the idea that we can take the exponents of the base-2 expansion of the position and put there the list of the elements of $\mathcal{A}$ whose $\mathcal{A}$-numbers are precisely these exponents.

This idea was prolific, because in the obtained copy, the image of the addition function is primitive recursive. Moreover, iterating the construction twice, produces a copy in which the image of the multiplication function is primitive recursive as well.

Then it is natural to hope that another iteration would make the image of the exponential function primitive recursive. But that is not the case: in the obtained copy the image of the function $2^x$ is primitive recursive if and only if it already is in the original copy $\mathcal{A}$.

More formally: given any punctual copy $\A = (\Nat, S^\A)$ of $(\Nat, S)$ with unique isomorphism $c$ from $(\Nat, S)$ to $\A$
we define a new copy $\widetilde{\A}$ by the following isomorphism $\widetilde{c}$:
$$ \widetilde{c}(0) \;=\; 0, \;\;\;\; \widetilde{c}(2^{i_k} + 2^{i_{k-1}} + \ldots + 2^{i_0}) \;=\; 2^{c(i_k)} + 2^{c(i_{k-1})} + \ldots + 2^{c(i_0)}. $$
So the individuals in $\widetilde{\A}$ are regarded as having the form $\widetilde{c}(n)$ for $n\in\Nat$ and the $\widetilde{\A}$-number of $\widetilde{c}(n)$ is $n$.
Therefore, the successor $S^{\widetilde{\A}}$ satisfies $S^{\widetilde{\A}}(\widetilde{c}(n)) = \widetilde{c}(n+1)$.

\begin{lemma} The function $add : \Nat\times\Nat \to \Nat$, defined by the equality
  $$ add(x,u) \;=\; \widetilde{c}(n + 2^j), $$
  where $x = \widetilde{c}(n)$ and $u = c(j)$ is primitive recursive.
\end{lemma}
\begin{proof} Clearly, the function $add$ realizes addition of $x$ and $2^{c(j)}$ in the model $\widetilde{\A}$. The idea of the proof will be to simulate binary addition in $\widetilde{\A}$. We will need the elementary relation $bit$, such that $bit(x,i) = true$ if and only if the $i$-th bit in the binary representation of $x$ is $1$ (equivalently, $2^i$ belongs to the binary representation of $x$).

  Given input $x, u$:

  1. Search for the least $k \leq x$, such that $bit(x, (S^\mathcal{A})^k(u)) = false$.

  2. Let $u_i = (S^\mathcal{A})^i(u)$ for $i = 0, 1, \ldots, k$.

  3. Return answer $x + 2^{u_k} - \sum_{i < k} 2^{u_i}$.\newline

  Observe that if we replace $S^\A$ by the standard $S$ (in all places where it is used) we obtain the standard addition of $x$ and $2^u$.
  Indeed, we can scan the binary representation of $x$, searching for the first $0$ digit to the right of $u$ and then change all the scanned $1$-s to $0$-s.
  The greatest possible value of $k$ is the number of $1$-s in the binary representation of $x$, therefore $k \leq x$.

  Since the algorithm works correctly using the standard $S$, it works correctly in the model $\widetilde{\A}$ using the successor $S^\A$.
\end{proof}

\begin{proposition}\label{propSsum} The images $S^{\widetilde{\A}}$ and $+^{\widetilde{\A}}$ of the successor and the addition in the model $\widetilde{\A}$ are primitive recursive.
\end{proposition}
\begin{proof}
  We have the equality $S^{\widetilde{\A}}(x) = add(x,c(0))$, therefore $S^{\widetilde{\A}}$ is primitive recursive.
  As for the sum $x+^{\widetilde{\A}}y$, let us represent $y$ in the form 
   $$y = 2^{u_k} + 2^{u_{k-1}} + \ldots + 2^{u_0}$$
  (clearly $x+^{\widetilde{\A}}y = x$ when $y = 0$). Then we can iterate the function $add$ in the following way:
  $$ x+^{\widetilde{\A}}y \;=\; add(\ldots add(add(x, u_0),u_1) \ldots u_k). $$
  The number of iterations is at most (the binary length of) $y$.
\end{proof}

By a very similar analysis, it can be shown that whenever the ordering $<^\A$ is primitive recursive, the ordering $<^{\widetilde{\A}}$ is also primitive recursive.

\begin{proposition}\label{main1} Let $\A$ be the model, constructed in the beginning of Section \ref{sec contribution 1a technical}. Then $S^{\widetilde{\A}}$ and $+^{\widetilde{\A}}$ are primitive recursive, but for any primitive recursive $f$, such that $f(x) \geq x^2$ for all sufficiently large $x$, its image $f^{\widetilde{\A}}$ is not primitive recursive.
\end{proposition}
\begin{proof} Since $S^\A$ is primitive recursive, it follows from Proposition \ref{propSsum} that $S^{\widetilde{\A}}$ and $+^{\widetilde{\A}}$ are primitive recursive.
  Now let $f$ be primitive recursive, such that $f(x) \geq x^2$ for almost all $x$ and assume $f^{\widetilde{\A}}$ is also primitive recursive. The proof follows the same argument as in Proposition \ref{propsq}.

  We will define a new function $g^\A$ and then we will argue that its preimage $g$ is primitive recursive and $g(p) \geq 2p$ for almost all $p$, which is impossible by Proposition \ref{prop2x}.

  In order to compute $g^\A(x)$: take $y = f^{\widetilde{\A}}(2^x)$ and then give answer $c(n)$, where $2^{c(n)}$ belongs to the binary representation of $y$ and $n$ is maximal.

  Observe that we can write $y = 2^{u_k} + 2^{u_{k-1}} + \ldots + 2^{u_0}$ and we can determine $c(n)$ with maximal $n$ using the relation $<^\A$, which is primitive recursive according to Lemma \ref{ineq}. Thus $g^\A$ is primitive recursive.

  If $x$ has sufficiently large $\A$-number $p$, then $2^x$ has $\widetilde{\A}$-number $2^p$, therefore $y = f^{\widetilde{\A}}(2^x)$ has $\widetilde{\A}$-number $f(2^p) \geq 2^{2p}$. So the maximal $n$, such that $2^{c(n)}$ belongs to the binary representation of $y$ must be greater than $2p$. In other words, $g^\A(x) = c(n)$ has $\A$-number $n \geq 2p$, that is $g(p) \geq 2p$.

  In addition, the $\A$-number $g(p) = n$ of $c(n)$ can be computed trivially from the $\widetilde{\A}$-number $f(2^p)$ of $y = f^\mathcal{C}(2^x)$.
  Since $f$ is primitive recursive, $g$ is also primitive recursive.
\end{proof}

Now we turn to multiplication.

\begin{proposition}\label{propprod}
  Let the image $+^\A$ of addition in the original model $\A$ be primitive recursive.
  Then the image $\cdot^{\widetilde{\A}}$ of multiplication in the model $\widetilde{\A}$ is also primitive recursive.
\end{proposition}
\begin{proof}
  Similar to addition, first we define the function $prod : \Nat\times\Nat \to \Nat$, such that
  $prod(x,u) = \widetilde{c}(n \cdot 2^j)$, where $x = \widetilde{c}(n)$ and $u = c(j)$. Let us represent
  $$ n = 2^{i_k} + 2^{i_{k-1}} + \ldots + 2^{i_0}. $$
  Then $n \cdot 2^j = 2^{i_k + j} + 2^{i_{k-1} + j} + \ldots + 2^{i_0 + j}$ and
  \begin{align*} prod(x,u) &= 2^{c(i_k + j)} + 2^{c(i_{k-1} + j)} + \ldots + 2^{c(i_0 + j)}\\
                &= 2^{c(i_k)+^\A c(j)} + 2^{c(i_{k-1}) +^\A c(j)} + \ldots + 2^{c(i_0)+^\A c(j)}
  \end{align*}
  Of course, this implies that $prod(x,u)$ can be computed primitive recursively from $x = 2^{c(i_k)} + 2^{c(i_{k-1})} + \ldots + 2^{c(i_0)}$ and $u = c(j)$
  using the fact that $+^\A$ is primitive recursive.

  And now it remains to iterate $prod$ using $+^{\widetilde{\A}}$ (which is primitive recursive due to Proposition \ref{propSsum}). Let $y = 2^{u_k} + 2^{u_{k-1}} + \ldots + 2^{u_0}$. Then
  $$ x\cdot^{\widetilde{\A}}y \;=\; prod(x, u_0) +^{\widetilde{\A}} prod(x, u_1) +^{\widetilde{\A}} \ldots +^{\widetilde{\A}} prod(x, u_k) $$
  and clearly $x\cdot^{\widetilde{\A}}0 = 0$.
\end{proof}

For the next proposition we will iterate the construction twice.

\begin{proposition}\label{main2} Let $\A$ be the model, constructed in the beginning Section \ref{sec contribution 1a technical}. Then $+^{\widetilde{\widetilde{\A}}}$ and $\cdot^{\widetilde{\widetilde{\A}}}$ are primitive recursive, but for any primitive recursive $f$, such that $f(x) \geq x^{\log_2 x}$ for all sufficiently large $x$, its image $f^{\widetilde{\widetilde{\A}}}$ is not primitive recursive.
\end{proposition}
\begin{proof} Let us denote $\B = \widetilde{\A}$ and $\C = \widetilde{\B} = \widetilde{\widetilde{A}}$. We know that $S^\B$ and $+^\B$ are primitive recursive, so by Propositions \ref{propSsum} and \ref{propprod},
  $S^\C$, $+^\C$ and $\cdot^\C$ are primitive recursive.

  Now let $f$ be primitive recursive, such that $f(x) \geq x^{\log_2 x}$ for almost all $x$, and assume $f^\C$ is also primitive recursive.

  We will define $g^\B$ and then we will argue that its preimage $g$ is primitive recursive and $g(p) \geq p^2$, which is impossible by Proposition \ref{main1}.

  To compute $g^\B(x)$: take $y = f^\C(2^x)$ and give answer $\widetilde{c}(i_k + \ldots + i_0)$, where $y = 2^{\widetilde{c}(i_k)} + 
  \ldots 2^{\widetilde{c}(i_0)}$.

  Observe that $\widetilde{c}(i_k + \ldots + i_0) = \widetilde{c}(i_k) +^\B \ldots +^\B \widetilde{c}(i_k)$, which can be computed from the binary representation of $y$ and the primitive recursive function $+^\B$. Thus $g^\B$ is primitive recursive.

  If $x$ has sufficiently large $\B$-number $p$, then $2^x$ has $\C$-number $2^p$, therefore $y = f^\C(2^x)$ has $\C$-number $f(2^p) \geq 2^{p^2}$. So the sum $i_k + \ldots + i_0$ of the exponents in the binary expansion of $f(2^p)$ is greater than $p^2$. In other words, $g^\B(x) = \widetilde{c}(i_k + \ldots + i_0)$ has $\B$-number $i_k + \ldots + i_0 \geq p^2$, that is $g(p) \geq p^2$.

  In addition, the $\B$-number $g(p) = i_k + \ldots + i_0$ can be computed trivially from the $\C$-number $f(2^p)$ of $y = f^\C(2^x)$.
  Since $f$ is primitive recursive, $g$ is also primitive recursive.
\end{proof}

We conclude this section by proving that the same idea does not work for exponentiation.

\begin{proposition}\label{main3}
  Let the image $+^\A$ of addition in the original model $\A$ be primitive recursive.
  Let $p(x) = 2^x$. Then $p^{\widetilde{\A}}$ is primitive recursive if and only if $p^\A$ is primitive recursive.
\end{proposition}
\begin{proof} Let $p^{\widetilde{\A}}$ be primitive recursive. We have the equality $p^{\widetilde{\A}}(2^{c(x)}) = 2^{c(2^x)}$, because the $\widetilde{\A}$-number of $2^{c(x)}$ is $2^x$ and the $\widetilde{\A}$-number of $2^{c(2^x)}$ is $p(2^x)$. So given $c(x)$ we can primitive recursively compute $c(2^x)$, which means that $p^\A$ is primitive recursive.

  Conversely, let $p^\A$ be primitive recursive. Then for $x = 2^{c(i_k)} + \ldots + 2^{c(i_0)}$ we have
  \begin{align*}
   p^{\widetilde{\A}}(x) &\;=\; 2^{c(2^{i_k} \;+\; \ldots \;+\; 2^{i_0})}\\
                    &\;=\; 2^{\left(c(2^{i_k}) \;+^\A\; \ldots \;+^\A\; c(2^{i_0})\right)}\\
                    &\;=\; 2^{\left(p^\A(c(i_k)) \;+^\A\; \ldots \;+^\A\; p^\A(c(i_0))\right)}.
  \end{align*}
  We can compute $c(i_k), \ldots, c(i_0)$ from the binary expansion of $x$ and since $+^\A$ and $p^\A$ are primitive recursive,
  it follows that $p^{\widetilde{\A}}$ is also primitive recursive.  
\end{proof}

\section{Mainland--Island Technique and the Metatheorem}\label{sec contribution 1b technical}

\subsection{Showcasing the Mainland--Island Technique}
\thmPrimes*
\begin{proof}
    For simplicity, the proof is given for the set of primes $\{2,3\}$. Generalization to arbitrary finite set of primes is easy---we briefly discuss it later.
    
    We construct a punctual model $\mathcal B$ and an isomorphism $c: (\mathbb N,S) \to \mathcal B$. At each stage, any element of $\mathcal B$ has a unique mark---a linear (possibly constant) function of the form $dx + e$, where $d = 0$ or $d = 2^u3^v$, for some $u,v,e \in \mathbb N$. The current domain of $c$ is the set of elements with constant functions as marks; these elements will always constitute an initial segment of $\mathbb N$. We maintain that the structure $(range(c), S^\mathcal{B})$ is an  $S^\mathcal{B}$-chain called the mainland and $c$ is the unique isomorphism from $(dom(c), S)$ to $(range(c),S^\mathcal{B})$. Elements in $\mathcal B \setminus range(c)$ are arranged in disjoint $S^\mathcal{B}$-chains (c.f., Figure \ref{fig:chains}) and are called islands. The initial element of an island will always have $2^u3^vx$ as a mark, for some $u,v \in \mathbb N$.

    Let $R_0, R_1, \dots$ be a p.r.\ enumeration of all codes of triples $\langle a,b,n \rangle$ such that $a$ has a prime factor $p \notin \{2,3\}$, and $b,n \in \mathbb N$. Each $R_i = \langle a,b,n \rangle$ stands for the following requirement: $(ax + b)^\mathcal{B} \neq \psi_i(x)$, where $\psi_i$ is the $i$th p.r.\ unary function according to a fixed computable enumeration of such functions. We will satisfy requirements one after another. Each requirement will go through the following states: inactive, active and satisfied. When $R_i = \langle a,b,n \rangle$ becomes active, it is associated with a fresh witness $w$ which starts its own island in $\mathcal B$. The aim is to make $(ax+b)^\mathcal{B}(w) \neq \psi_n(w)$. We wait for a stage $s$ at which $\psi_{n,s} (w)\downarrow$, each time quickly extending the current structure to make the images of $S,2x,3x$ p.r.\ in $\mathcal B$. These extensions give rise to a growing collection of islands. Once we find such an $s$, we run \emph{Connect}---a procedure that satisfies $R_i$ by appropriately merging all the islands into the mainland. The details are given below.

 When we add a new element into $\mathcal B$ it is always the least number currently not in $\mathcal B$ (so-called fresh number).

\begin{figure}
    \centering

\begin{tikzpicture}[scale=0.8]
\foreach \x [count=\i] in {0,2,...,10} {

\ifnum\x>0
    
    \ifnum\x=8
            \draw[loosely dotted, line width=0.5pt] (\x-0.6,0) -- (\x-1.6,0);
    \else
        \draw[<-, line width=0.5pt] (\x-0.6,0) -- (\x-1.6,0);
    \fi 
    \fi
}

\foreach \x [count=\i] in {0,2,...,6} {

    \filldraw[black] (\x,0) circle (0.75pt) node[above] {\pgfmathprint{int( \x / 2)}};

}

    \filldraw[black] (8,0) circle (0.75pt) node[above] {$110$};
    \filldraw[black] (10,0) circle (0.75pt) node[above] {$111$};

\foreach \x [count=\i] in {0,2,...,10} {

    \ifnum\x=0
        \filldraw[black] (\x,-1.5) circle (0.75pt) node[above] {$x$};
    \else
        \filldraw[black] (\x,-1.5) circle (0.75pt) node[above] {$x + $\pgfmathprint{int( \i - 1 )}};
    \fi

    \ifnum\x>0
        \draw[<-, line width=0.5pt] (\x-0.6,-1.5) -- (\x-1.6,-1.5);
    \fi
}

\foreach \x [count=\i] in {0,2,...,10} {
    \ifnum\x=0
        \filldraw[black] (\x,-3) circle (0.75pt) node[above] {$2x$};
    \else
    \filldraw[black] (\x,-3) circle (0.75pt) node[above] {$2x + $\pgfmathprint{int( \i - 1 )}};
    \fi

    \ifnum\x>0
        \draw[<-, line width=0.5pt] (\x-0.6,-3) -- (\x-1.6,-3);
    \fi
}

\foreach \x [count=\i] in {0,2,...,10} {
        \ifnum\x=0
        \filldraw[black] (\x,-4.5) circle (0.75pt) node[above] {$3x$};
    \else
            \filldraw[black] (\x,-4.5) circle (0.75pt) node[above] {$3x + $\pgfmathprint{int( \i - 1 )}};
    \fi

    \ifnum\x>0
        \draw[<-, line width=0.5pt] (\x-0.6,-4.5) -- (\x-1.6,-4.5);
        
    \fi
}

\end{tikzpicture}

    \caption{Visualization of the mainland (top chain) and three islands with corresponding marks. Arrows correspond to $S^\mathcal B$.}
    \label{fig:chains}
\end{figure}
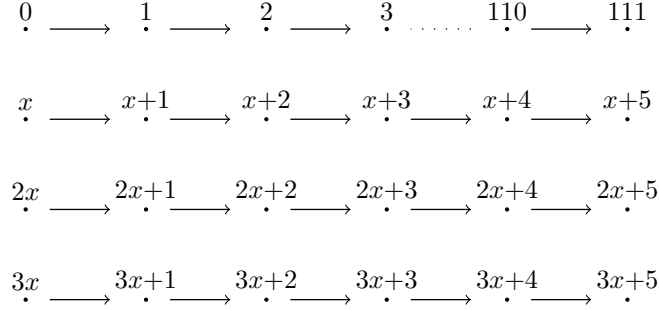

\paragraph*{Construction.}
    At stage $0$, $\mathcal B$ consists of $0$ with mark $0$, $S^\mathcal B =\emptyset$ and $c(0)=0$, so the mainland contains just one element $0$. All requirements are inactive. 
    
    Let $s > 0$. First, we run \emph{Extend} (see below) and obtain $\mathcal B'$.
    
    If no requirement is active, let $i$ be the least number such that $R_i$ is inactive. $R_i$ becomes active. We add a fresh number $w$ to $\mathcal B'$. $w$ becomes a witness for $R_i$ and receives a mark $x$. We set $\mathcal B_s = \mathcal B'$ and go to the next stage. 

    If a requirement is active, let $R_i = \langle a,b,n \rangle$ be the one and let $w$ be its witness. If $\psi_{n,s}(w)\!\uparrow$, we set $\mathcal B_s = \mathcal B'$ and go to the next stage. Otherwise we execute \emph{Connect}$(\mathcal B', \psi_{n,s}(w))$. This yields a model that becomes $\mathcal B_s$. We go to the next stage.

    \paragraph*{$Extend$} We are given $\mathcal B,c,M(\mathcal B)$. We need $\mathcal B' \supset \mathcal B$, $c' \supset c$ and $M(\mathcal B')$ with the following properties:
    \begin{enumerate}[label=\arabic*.]
        \item $M(\mathcal B') \supset M(\mathcal B)$,
        \item If $dx+e \in M(\mathcal B)$, then $dx+e+1, 2(dx+e),3(dx +e) \in M(\mathcal B')$
        \item If $dx + e, dx+e'\in M(\mathcal B')$ and $e < e'$ then $dx + e'' \in M(\mathcal B')$, for every $e'' \in \mathbb N$ such that $e < e'' < e'$.

    \end{enumerate}
We take the least set of marks that satisfy the above conditions---this is $M(\mathcal B')$. We extend $\mathcal B$ to $\mathcal B'$ such that $|\mathcal B'| = |M(\mathcal B')|$ and declare that for every $b,b' \in \mathcal B'$, $S^{\mathcal B'}(b) = b'$ iff, for some $d,e \in \mathbb N$, the mark of $b$ is $dx+e$ and the mark of $b'$ is $dx+e + 1$.

    \paragraph*{$Connect(\mathcal B',m).$} Suppose that the current requirement has been active for $t$ stages $s_0, s_1, \dots, s_t$, where $s_t$ is the current stage. The number of islands in $\mathcal B'$ is at most $\Sigma _{j=0}^{t} (j+1)$, because at stage $s_j$ we add new islands with marks $2^u3^v x$ such that $u+v = j$ (cf. \emph{Extend}). Clearly, this gives us a uniform primitive recursive bound on the number of islands at any given stage. Let the number of islands in $\mathcal B'$ be $k$.

    Use a primitive recursive algorithm to arrange the islands in the linear order $I_1, I_2, \dots, I_k$ as follows: an island with mark $dx$ comes before an island with mark $d'x$ iff $d < d'$. Let $I_0$ be the mainland.

    Our aim is to satisfy $R_i$ by carefully connecting all the islands 

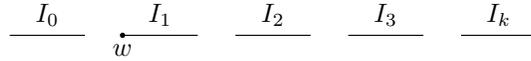
\begin{center}
\begin{tikzpicture}
\draw (6, 1) -- (7, 1) node[style={midway},above] {$I_k$};
\draw (0, 1) -- (1, 1) node[style={midway},above] {$I_0$};
\draw (1.5, 1) coordinate (w) -- (2.5, 1) node[style={midway},above] {$I_1$};
\draw (3, 1) -- (4, 1) node[style={midway},above] {$I_2$};
\draw (4.5, 1) -- (5.5, 1) node[style={midway},above] {$I_3$};

\fill (w) circle (1pt);
\node[below] at (w) {$w$};
\end{tikzpicture}

\end{center}
to one mainland
  
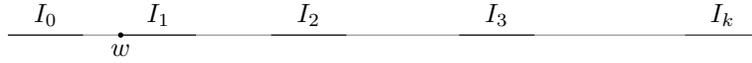
\begin{center}
  \begin{tikzpicture}
 \draw[gray] (0,0) -- (9,0);
\draw (9,0) -- (10,0)  node[style={midway},above] {$I_k$};
\draw (0,0) -- (1,0) node[style={midway},above] {$I_0$};
\draw (1.5,0) coordinate (w) -- (2.5,0) node[style={midway},above] {$I_1$};
\draw (3.5,0) -- (4.5,0) node[style={midway},above] {$I_2$};
\draw (6,0) -- (7,0) node[style={midway},above] {$I_3$};

\fill (w) circle (1pt);
\node[below] at (w) {$w$};

    \end{tikzpicture}
\end{center}
so that islands do not overlap and $a c^{-1}(w) + b \neq c^{-1}(m)$.

We need to find the right value for $c^{-1}(w)$. Let us denote it by $q$ for short. Once we find $q$ we may execute the procedure described in the paragraph below.

$Finish(q)$. We extend $S^{\mathcal B'}$, and possibly its domain, so that $[0,|\mathcal B'|)$ is the least initial segment of $\mathbb N$ such that
 $\mathcal B'$ is a single $S^{\mathcal B'}$-chain and, for every $dx+e \in M(\mathcal B')$, $c$ maps $d q + e$ to the element with mark $dx+e$, where $c$ is the unique isomorphism from $([0,|\mathcal B'|),S)$ to $\mathcal B'$. Notice that the extended $c$ maps $q$ to $w$, as $w$ has mark $x$.
We declare that $\mathcal B_s$ is the obtained structure, $c$ is the unique isomorphism described above and $M(\mathcal B_s) = dom(c)$, i.e., every element has a constant function as its mark.
    
Let us find $q$. Let $d_i x$ and $d_i x + e_i$ be the marks of the first and of the last element of $I_i$, respectively.

    Set $q' := |I_0|$. Execute the bounded loop for $i = 1,2,\dots,k-1$: $q' := max \{q', \lceil  \frac{e_i}{d_{i+1}-d_i} \rceil + 1 \}$. $\frac{e_i}{d_{i+1}-d_i}$ is just the solution of $d_i x + e_i = d_{i+1}x$. If we executed $Finish(q')$, we would guarantee that islands do not overlap, because $q'$ satisfies the set of inequalities $d_i x + e_i < d_{i+1}x$, for $i = 1,2,\dots, k-1$. To satisfy $R_i$, we may need to change $q'$ in certain cases.

\begin{enumerate}[label=\arabic*.]
    \item Case $m \in I_0$. Execute $Finish(q')$. $R_i$ is satisfied because $c^{-1}(m) < c^{-1}(w) = q < aq + b$.

    
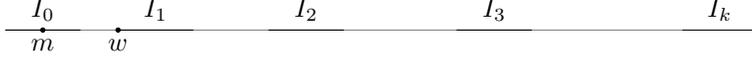
\medskip

    \begin{tikzpicture}

\draw[gray] (0,0) -- (9,0);
\draw (9,0) -- (10,0)  node[style={midway},above] {$I_k$};
\draw (0,0) -- (1,0) node[style={midway},above] {$I_0$};
\draw (1.5,0) coordinate (w) -- (2.5,0) node[style={midway},above] {$I_1$};
\draw (3.5,0) -- (4.5,0) node[style={midway},above] {$I_2$};
\draw (6,0) -- (7,0) node[style={midway},above] {$I_3$};

\draw (0.5,0) coordinate (m);
\fill (m) circle (1pt);
\node[below] at (m) {$m$};

\fill (w) circle (1pt);
\node[below] at (w) {$w$};

\end{tikzpicture}

\item Case $m \notin I_0 \cup I_1 \cup \dots \cup I_k$. Set $c(|I_0|) = m$ and $q = max(q', |I_0|)$. Call $Finish(q)$. $R_i$ is satisfied because $c^{-1}(m) < c^{-1}(w) = q < aq + b$.

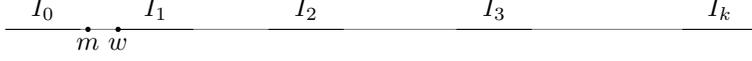
\begin{tikzpicture}

\draw[gray] (0,0) -- (9,0);
\draw (9,0) -- (10,0)  node[style={midway},above] {$I_k$};
\draw (0,0) -- (1,0) node[style={midway},above] {$I_0$};
\draw (1.5,0) coordinate (w) -- (2.5,0) node[style={midway},above] {$I_1$};
\draw (3.5,0) -- (4.5,0) node[style={midway},above] {$I_2$};
\draw (6,0) -- (7,0) node[style={midway},above] {$I_3$};

\draw[gray] (1.1,0) coordinate (m);  
\fill (m) circle (1pt);
\node[below] at (m) {$m$};

\fill (w) circle (1pt);
\node[below] at (w) {$w$};

\end{tikzpicture}

\item Case $m \in I_j, 1 \leq j \leq k, a < d_j$. Let the mark of $m$ be $d_jx + e$. Compute $q= max(q',\lceil \frac{b-e}{d_j - a}\rceil+1)$. Notice that $\lceil \frac{b-e}{d_j - a}\rceil$ is the least integer satisfying $ax + b \leq d_jx + e$. Call $Finish(q)$. $R_i$ is satisfied because $aq + b < d_j q + e = c^{-1}(m)$.

\item Case $m \in I_j, 1 \leq j \leq k, a > d_j$. Let the mark of $m$ be $d_jx + e$. Compute $q= max(q',\lceil \frac{e-b}{a - d_j}\rceil+1)$. Notice that $\lceil \frac{e-b}{a - d_j}\rceil$ is the least integer satisfying $ax + b \geq d_jx + e$. Call $Finish(q)$. $R_i$ is satisfied because $aq + b > d_j q + e = c^{-1}(m)$.

\end{enumerate}
Notice that the case $m \in I_j, 1 \leq j \leq k, a = d_j$ is not possible, because $a$ has a prime factor $p \notin \{2,3\}$, while $d_j = 2^u3^v$, for some $u,v \in \mathbb N$.
    
\paragraph*{Verification.} Let us argue that $\mathcal B$ is a punctual structure. Clearly, $\mathbb N$ is the domain of $\mathcal B$ since at each stage we properly extend the current $\mathcal B$ by least numbers not in the current $\mathcal B$. To compute $S^\mathcal{B}(a)$, for $a \in \mathbb N$, we need to compute $\mathcal B_{a+1}$ since $a$ is already present in $\mathcal B_a$ and if $S^\mathcal{B}(a)$ is not defined in $\mathcal B_a$, it will be defined in $\mathcal B_{a+1}$ (cf. \emph{Extend}). $\mathcal B_{a+1}$ can be computed by a loop bounded by $a+1$ that computes $\mathcal B_0, \mathcal B_1, \dots, \mathcal B_{a+1}$ in succession. We have already explained in the construction that operations performed at a given stage are executed by a primitive recursive algorithm. 

In a similar way, we argue that the images of $2x$ and $3x$ in $\mathcal B$ are p.r. For let $a \in \mathbb N$. Suppose we want to compute $(2x)^\mathcal{B}(a)$. We compute $\mathcal B_{a+1}$. Clearly, $a \in \mathcal B_{a+1}$. Let $dx + e$ be the mark of $a$ in $\mathcal B_{a+1}$. Compute $\mathcal B'$ at stage $a+1$ (i.e., we call \emph{Extend}). Let $b$ be the element that has the mark $2(dx+e)$. We claim that $b$ is the value we want, namely $(2x)^\mathcal{B}(a)$. 

Suppose $a$ is in the mainland at the beginning of stage $a+1$, i.e. its mark is actually $e$ ($d = 0$). In \emph{Extend}, we add the mark $2e$, if its not already there. By the definition of $\mathcal B'$, it must be the case that the number of $b$ is twice as the number of $a$. Now suppose that $a$ is not in the mainland at the beginning of stage $a+1$. At some stage we will execute \emph{Connect}. There, we set $c(dq+e) = a$ and $c(2(dq+e)) = b$ and, therefore, $b$ is as required. A similar argument works for $3x$.
\end{proof}

\subsection{Metatheorem}
First, we introduce the technical definitions needed to formulate the main result of this subsection (\cref{theo:nonstandardness-meta} below) that gives sufficient conditions for constructing a punctually nonstandard copy of $\Std = (\mathbb{N},\Succ)$ while preserving primitive recursiveness of a class of functions $\mathcal{F}$. Intuitively speaking, these conditions can be isolated via a careful analysis of the mainland--island technique and the peculiarities of the Levitz's class $\LL$ (recall \cref{def:Levitz}). More precisely, an analysis of the Levitz's paper~\cite{Levitz} (to be elaborated in Subsection~\ref{subsect:Levitz}) hints on the following: every function $f\in \LL$ admits a `nice' normal form, and we capture this feature of $\LL$ via \cref{{def:norm-form}} below. 


Let $f$ and $g$ be unary primitive recursive functions. We say that $g$ \emph{eventually dominates} $f$ (denoted by $f \preccurlyeq g$) if there exists $n_0\in\mathbb{N}$ such that
\begin{equation}\label{equ:domination-point}
    (\forall x \geq n_0) (f(x) \leq g(x)).
\end{equation}
We call an arbitrary number $n_0$ satisfying Eq.~(\ref{equ:domination-point}) an \emph{$(f,g)$-domination witness}. If $n_0$ additionally satisfies 
$(\forall x \geq n_0) (f(x) < g(x))$,
then we say that $n_0$ is a \emph{strict $(f,g)$-domination witness}.


    Let $\mathcal{Q} = \{ q_i(x_1,x_2,\dots, x_{r_i})\}_{i\in\mathbb{N}}$ be a uniform family of primitive recursive functions. Let $T_1(\mathcal{Q})$ be the set of all unary functions belonging to the compositional closure of $\mathcal{Q} \cup \{ I^{n}_m : 1\leq m \leq n\}$,
    where
    $I^{n}_m(x_1,\dots,x_n) = x_m$.

    We say that a set of functions $\mathcal{F} \subseteq T_1(\mathcal{Q})$ is \emph{listable} if there exists a computable list $(t_i(x))_{i\in\mathbb{N}}$ of terms in the signature $\mathcal{Q}$ such that $\mathcal{F}$ is equal to the set of all functions that are defined by the terms $t_i(x)$, $i\in\mathbb{N}$.

    \begin{remark}
    Notice that, in general, two different terms can define the same function: for example, for $\mathcal{Q} = \{ +, \cdot\}$ consider the terms
    $t_1(x) = (x+x) \cdot x$ and $t_2(x) = x\cdot x + x\cdot x$.
    \end{remark}

    
    In the definition below, `$L$' stands for both `listable' and `linearly ordered'. 

    \begin{definition}\label{def:L-fam}
        We say that a pair $(\mathcal{Q},\mathcal{F})$ is an \emph{$L$-family} if the following conditions are satisfied:
        \begin{enumerate}
            \item $\mathcal{F} \subseteq T_1(\mathcal{Q})$ is a listable family of unary functions.

            \item The family $\mathcal{F}$ contains the successor $\Succ(x)$ and $I^1_1(x) = x$.

            \item The family $\mathcal{F}$ is closed with respect to composition, i.e., if $f(x)$ and $g(x)$ both belong to $\mathcal{F}$, then $g(f(x))\in \mathcal{F}$. 

            \item The set of functions $\mathcal{F}$ is linearly ordered with respect to the domination order $\preccurlyeq$.

            \item Every function $f(x)$ from $\mathcal{F}$ is strictly increasing: i.e., if $x< y$, then $f(x) < f(y)$. 
        \end{enumerate}
    \end{definition}

    \begin{definition}\label{def:norm-form}
    Let $(\mathcal{Q},\mathcal{F})$ be an \emph{$L$-family}. 
    We say that a primitive resursive set $T \subseteq \mathbb{N}$ is a \emph{set of normal forms} for  $(\mathcal{Q},\mathcal{F})$ if the following conditions are satisfied:
    \begin{enumerate}
        \item Given an arbitrary term $t(x)$ (in the signature $\mathcal{Q}$) that defines a function from $\mathcal{F}$, one can primitively recursively compute a value $N(t) \in T$.

        \item There exists a primitive recursive linear order $\leq^T$ on $T$ such that for arbitrary terms $t_1(x)$ and $t_2(x)$ defining functions from $\mathcal{F}$, we have: 
        \begin{enumerate}
            \item $t_1(x) \preccurlyeq t_2(x)$ if and only if $N(t_1) \leq^T N(t_2)$. 

            \item If $t_1(x) \prec t_2(x)$, then one can primitively recursively find a strict $(t_1,t_2)$-domination witness.    
        \end{enumerate}
    \end{enumerate}

    \end{definition}

The main result of this section is the following:

\theoMeta*

\begin{proof}
    Our construction will produce a punctual structure $\mathcal{A}$ isomorphic to $\Std = (\mathbb{N},\Succ)$ with the following properties:
    \begin{itemize}
        \item for every function $f\in  \mathcal{F}$, its representation $f^{\mathcal{A}}$ is primitive recursive,

        \item the ordering $\leq^{\mathcal{A}}$ is primitive recursive,

        \item the representation of the predecessor function $\Pred^{\mathcal{A}}$ is \emph{not} primitive recursive (hence, the structure $(\mathbb{N},S^{\mathcal{A}})$ is punctually nonstandard). 
    \end{itemize}

    We will satisfy the following requirements: for $e\in\mathbb{N}$,
    \begin{description}
        \item[$\mathcal{R}_e$:] $\Pred^{\mathcal{A}} \neq \psi_e$, where $\psi_e(x)$ is the $e$-th unary primitive recursive function.
    \end{description}

    First, we establish a useful fact about listable families:    

    \begin{lemma}\label{lem:prim-rec-list-0}
        Let $\mathcal{F} \subseteq T_1(\mathcal{Q})$ be a listable family, and let $(t_i(x))_{i\in\mathbb{N}}$ be a computable list of terms that lists $\mathcal{F}$. Then there exists a primitive recursive list of terms $(t'_s(x))_{s\in\mathbb{N}}$ such that $(t'_s(x))_{s\in\mathbb{N}}$ lists $\mathcal{F}$.
    \end{lemma}
    \begin{proof}
        For $s\in\mathbb{N}$, we define the term $t'_s(x)$ and an ancillary value $p(s)$ by recursion. We put $t'_0(x) = t_0(x)$ and $p(0) = 1$.

        For $s\geq 1$, if the term $t_{p(s-1)}(x)$ can be computed in at most $s$ computational steps, then we put $t'_s(x) = t_{p(s-1)}(x)$ and $p(s) = p(s-1)+1$. Otherwise, set $t'_s(x) = t_0(x)$ and $p(s) = p(s-1)$.
    \end{proof}
    
    By \cref{lem:prim-rec-list-0}, we choose a primitive recursive list of terms $(t_i(x))_{i\in\mathbb{N}}$ that lists our family~$\mathcal{F}$. 


    For given $e,s\in\mathbb{N}$ we can primitively recursively compute the following finite partial structure $\mathcal{M}{[e,s]}$. The signature of $\mathcal{M}[e,s]$ is 
    $\sigma_e = \{\Succ, \leq, t_0,t_1,\dots,t_e\}$.
         
    We put $\mathcal{M}[e,0] = \{0\}$. Intuitively speaking, $\mathcal{M}[e,s+1]$ is obtained from $\mathcal{M}[e,s]$ by applying all the functions from $\sigma_e$, plus taking the downwards closure in $\mathbb{N}$.
    More formally, the partial structure $\mathcal{M}{[e,s+1]}$ is defined as follows:    
    \begin{itemize}    
        \item The domain of $\mathcal{M}[e,s+1]$ contains all numbers $a\in\mathbb{N}$ such that
        $a\leq f(b)$, for some $b\in\operatorname{dom}(\mathcal{M}[e,s])$ and $f \in \sigma_e$.

        \item The ordering $\leq^{\mathcal{M}[e,s+1]}$ is the standard ordering induced from $\mathbb{N}$.

        \item For a function $f\in\sigma_e$ and $x\in \operatorname{dom}(\mathcal{M}[e,s])$, we have
        $f^{\mathcal{M}[e,s+1]}(x) = f(x)$.

        \item For $f\in\sigma_e$ and $x\in \operatorname{dom}(\mathcal{M}[e,s+1]) \setminus \operatorname{dom}(\mathcal{M}[e,s])$, the value $f^{\mathcal{M}[e,s+1]}(x)$ is undefined.
    \end{itemize}


    The desired punctual structure $\mathcal{A} \cong (\mathbb{N},\Succ)$ is built in stages: at a stage $s$ we construct a finite structure $\mathcal{A}_s$. We will have 
    $\operatorname{card}(\mathcal{A}_{s+1}) > \operatorname{card}(\mathcal{A}_s) \geq s+2$.
    We assume that $\mathcal{A}_0$ contains numbers $0$ and $1$, and $\Succ^{\mathcal{A}}(0) = 1$.

    At a given stage $s$, let $e(s)$ be the least index of a requirement $\mathcal{R}_{e(s)}$ that is not satisfied yet. We will also define an ancillary value $\theta(s)\in\mathbb{N}$ such that $\theta(s+1) > \theta(s)$.
    
    We will construct an isomorphic embedding $c[s]$ from the signature-reduct $\mathcal{M}[e(s),\theta(s)]\upharpoonright \{\Succ\}$ into the constructed finite $\mathcal{A}_s$. Sometimes we abuse our notations and we identify the reduct $\mathcal{M}[e(s),\theta(s)]\upharpoonright \{\Succ\}$ with $\mathcal{M}[e(s),\theta(s)]$.

    \begin{convention}
    The substructure 
    $\widehat{\mathcal{M}}[s] = (\operatorname{range}(c[s]),\Succ^{\mathcal{A}_s})$
    is called the \emph{mainland} (at the stage $s$), and the substructure 
    $\operatorname{Ar}[s] = \mathcal{A}_s \setminus \widehat{\mathcal{M}}[s]$
    is called the \emph{archipelago}. 
    
    We always assume that every element $x$ from the mainland is strictly $<^{\mathcal{A}}$-less than all elements from the archipelago.
    \end{convention}

    We note the following property of the mainland: if $x \in \widehat{\mathcal{M}}[s]$, then the element $\Succ^{\mathcal{A}}(x)$ belongs to $\widehat{\mathcal{M}}[s+1]$.

    We will ensure the following: if $x\in \mathcal{A}_s$, then there exists a stage $s'\geq s$ such that $x\in\widehat{\mathcal{M}}[s']$. That is, every element $x$ eventually joins the mainland. This property will guarantee that the constructed structure $\mathcal{A} = \bigcup_{s\in\mathbb{N}}\mathcal{A}_s$ is isomorphic to $(\mathbb{N},\Succ)$ via the map $c = \bigcup_{s\in\mathbb{N}} c[s]$.

    Whenever a fresh element $x$ is added into $\mathcal{A}_s$, $x$ is chosen as the least element currently not in $\mathcal{A}_s$.

\subparagraph*{Strategy for Requirement $\mathcal{R}_e$.}

    Suppose that $\mathcal{R}_e$ starts working at a stage $s_0\geq 1$.

    Firstly, choose a fresh witness $w$. The witness $w$ becomes an element of the (current) archipelago, so at stage $s_0$ it is not connected via the function $\Succ^{\mathcal{A}}$ to the mainland.

    Wait for a stage $s' > s_0$ such that $\psi_{e,s'}(w) \!\downarrow$. While waiting, we construct $\mathcal{A}_s$ via the extension procedure $\texttt{Extend}(e,s)$ described below. Here we note that $w$ stays not $\Succ^{\mathcal{A}}$-connected to the mainland. 

    In the procedure $\texttt{Extend}(e,s)$, we define an ancillary parameter $\beta(s)\in\mathbb{N}$. This parameter will help us dealing with the growth of the archipelago (to be elaborated below). We put $\beta(s_0) = s_0$.

    Suppose that $\psi_{e,s'}(w) \!\downarrow\ = y$. Then by Eq.~(\ref{equ:max}), we have $y < s'$ and hence, $y\in \mathcal{A}_{s'-2}$: indeed, note that $\operatorname{card}(\mathcal{A}_{s'-2}) \geq s'$ and thus, $\{0,1,\dots,s'-1\} \subseteq \operatorname{dom}(\mathcal{A}_{s'-2})$. Hence, $\Succ^{\mathcal{A}}(y) \in \mathcal{A}_{s'-1}$  and 
    $\Succ^{\mathcal{A}}(y) \neq w$. 
    Indeed, this is a consequence of the following properties:
    \begin{itemize}
        \item In the structure $\mathcal{A}_{s'-1}$ the element $w$ is not connected to the mainland, and thus, $w\neq \Succ^{\mathcal{A}}(z)$ for all $z\in\widehat{\mathcal{M}}[s'-1]$.

        \item If $z \in \operatorname{Ar}[s'-1] = \mathcal{A}_{s'-1} \setminus \widehat{\mathcal{M}}[s'-1]$, then (by the description of the extension procedure $\texttt{Extend}(e,s)$ below) there exists some term $q(x)$ in the signature $\sigma_e$ such that $z = q^{\mathcal{A}}(w)$ and $q(x)\in\mathcal{F}$. Since every function from the family $\mathcal{F}$ is strictly increasing (Item~(5) of Definition~\ref{def:L-fam}), we deduce that $w \leq^{\mathcal{A}} z$ and $w\neq \Succ^{\mathcal{A}}(z)$.
    \end{itemize}

    So, we have $\Succ^{\mathcal{A}}(\psi_e(w)) \neq w$, and thus we deduce that $\psi_e(w) \neq \Pred^{\mathcal{A}}(w)$. Therefore, the requirement $\mathcal{R}_e$ is (forever) satisfied. 

    At the end of the stage $s'$ we connect the current archipelago to the mainland via the connection procedure $\texttt{Connect}(e,s')$ described below. After that, we can safely move to the next unsatisfied requirement.

\subparagraph*{Procedure $\texttt{Extend}(e,s)$.} 
    Firstly, we put $\theta(s) = \theta(s-1) + 1$, and we use fresh numbers to produce the following extensions:
    \begin{itemize}
        \item the mainland $\widehat{\mathcal{M}}[s-1]$ is extended to the new mainland $\widehat{\mathcal{M}}[s]$ that is isomorphic to $\mathcal{M}[e,\theta(s-1)+1]$,

        \item the isomorphism $c[s-1]$ (acting from $\mathcal{M}[e,\theta(s-1)]$ onto $\widehat{\mathcal{M}}[s-1]$) is extended to an isomorphism $c[s]$ from $\mathcal{M}[e,\theta(s-1)+1]$ onto $\widehat{\mathcal{M}}[s]$.
    \end{itemize}

    Along the construction, every element $z$ from the archipelago gets its own label $\boxed{q_z(w)}$, where $q_z(x)$ is a term of the signature $\sigma_e = \{\Succ, t_0,t_1,\dots,t_e\}$ such that $q_z(x)$ defines a function from the family $\mathcal{F}$. The label satisfies $q^{\mathcal{A}}_z(w) = z$.
    
    We assume that the element $w$ has label $\boxed{w}$ (recall that the function $I^1_1$ belongs to $\mathcal{F}$). 

    We fix the set of normal forms $T$ from Definition~\ref{def:norm-form} (and the corresponding objects $N(t)$, $\leq^T$).

    Suppose that we have already defined labels $\boxed{q_z(w)}$ for each element $z\in \operatorname{Ar}[s-1]$.
    Then the new archipelago $\operatorname{Ar}[s] \supset \operatorname{Ar}[s-1]$ is constructed as follows.
    
    One-by-one, for each function $f\in\sigma_e$ and each label $\boxed{q_z(w)}$ for $z\in \operatorname{Ar}[s-1]$, we proceed with the following actions:
    \begin{itemize}
        \item Consider the term $u(x) = f(q_z(x))$. By Items~(2)--(3) of Definition~\ref{def:L-fam}, $u(x)$ defines a function from the family $\mathcal{F}$. Hence, we can primitively recursively compute the value $N(u)$.

        \item If there exists an element $a$ from the current archipelago such that $N(q_{a}) = N(u)$ (that is, $q_{a}(x)$ and $u(x)$ define the same function), then we declare that $f^{\mathcal{A}}(z) = a$. This declaration is well-defined, since we must have
        \[
            a = q^{\mathcal{A}}_a(w) = u^{\mathcal{A}}(w) = f^{\mathcal{A}}(q_z^{\mathcal{A}}(w)) = f^{\mathcal{A}}(z).
        \]

        \item Otherwise, for all $a$ from the current archipelago we have $N(q_a) \neq N(u)$. Since $\leq^{T}$ is a linear order on $T$ (and the set $\mathcal{F}$ is linearly ordered with respect to the domination order $\preccurlyeq$), this implies that for every $a$, we have either $q_a(x) \prec u(x)$ or $u(x) \prec q_a(x)$.
        
        We add a fresh element $a^{\ast}$ to the archipelago and define its label  $\boxed{q_{a^{\ast}}(w)} = \boxed{u(w)}$. We put $f^{\mathcal{A}}(z) = a^{\ast}$. In addition, for each element $b\neq a^{\ast}$ from the current archipelago we make the following actions:
        \begin{itemize}
            \item If $N(q_{b}) <^T N(u)$, then we compute a strict $(q_{b},u)$-domination witness $d_{b}$. We also declare that $b <^{\mathcal{A}} a^{\ast}$.

            \item If $N(u) <^T N(q_{b})$, then we compute a strict $(u,q_{b})$-domination witness $d_{b}$. We declare that $a^{\ast} <^{\mathcal{A}} b$.
        \end{itemize}
    \end{itemize}
    
    In the end, the value $\beta(s)$ is defined as $1+$(the maximum of $\beta(s-1)$ and all the strict domination witnesses $d_{b}$ that appear in the computations above). Consequently, if $x \geq \beta(s)$ and $a,b\in \operatorname{Ar}[s]$ and $a <^{\mathcal{A}} b$, then we have $q_a(x) < q_b(x)$.

\subparagraph*{Procedure $\texttt{Connect}(e,s)$.}
    We define the value $\Theta =  \max(\theta(s-1),\beta(s-1))$. The computation of the value $\beta(s-1)$ given above ensures the following property of the parameter $\Theta$: 
    \begin{description}
    \item[($\dagger$)] Suppose that at the beginning of the stage $s$ we have two elements $a \neq b$ from the archipelago $\operatorname{Ar}[s-1]$. Then the corresponding labels $\boxed{q_a(w)}$ and $\boxed{q_b(w)}$ satisfy the following: if $x\geq \Theta$ and $a<^{\mathcal{A}} b$, then $q_a(x) < q_b(x)$.
    \end{description}
    We compute the values 
    $D' = \max(\operatorname{dom}(\mathcal{M}[e,\theta(s-1)]))$  and $D = c(D')$.

    If $D' < \Theta$, then firstly we use fresh numbers to produce an extension of $\widehat{\mathcal{M}}[s-1]$ to an isomorphic copy of $\mathcal{M}[e,\Theta]$ (and we also extend the isomorphism $c[s-1]$ in appropriate way). Thus, without loss of generality, we may assume that $D' \geq \Theta$.

    Note that the value $\Succ^{\mathcal{A}_{s-1}}(D)$ is undefined. Thus, our connection procedure (that builds $\mathcal{A}_{s}$) is arranged as follows:
    \begin{enumerate}[label=(\alph*)]
        \item We put $\Succ^{\mathcal{A}}(D) = w$, and we set $c(D'+1) = w$.

        \item For every $a\in \operatorname{Ar}[s-1]$ such that $a\neq w$, we find the label $\boxed{q_a(w)}$ and we define
        $c(q^{\Std}_a(D'+1)) = a$, 
        where the value $q^{\Std}_a(D'+1)$ is evaluated in the standard copy $\Std = (\mathbb{N},\Succ)$. 
        
        Since every function $q_a(x) \in \mathcal{F}$ is strictly increasing, we have $q^{{\Std}}_a(D'+1) \geq D' + 1$ and thus $q^{{\Std}}_a(D'+1)$ is greater than all elements from  $\mathcal{M}[e,\theta(s-1)]$.

        By Property~($\dagger$) we deduce that all $a,b\in \operatorname{Ar}[s-1]$ satisfy: 
        \[
        q^{{\Std}}_a(D'+1) \leq q^{{\Std}}_b(D'+1)\ \Longleftrightarrow\ a\leq^{\mathcal{A}} b.
        \]
        In particular, $c^{-1}\upharpoonright \operatorname{Ar}[s-1]$ is a well-defined map that is an isomorphism between the linear orders $(\operatorname{Ar}[s-1],\leq^{\mathcal{A}})$ and $(c^{-1}(\operatorname{Ar}[s-1]), \leq)$.

        We obtain that (at the moment) we have
        \[
            x \leq y\ \Leftrightarrow\ c(x) \leq^{\mathcal{A}} c(y),\ \text{ for all } x,y\in \operatorname{dom}(c[s]).
        \]

        \item We compute $\Theta' = \max\{ c^{-1}(a) : a \in \operatorname{Ar}[s-1]\}$. Note that $\Theta' \geq D' + 1 > \Theta \geq \theta(s-1)$. Hence, $\Theta' \geq \theta(s-1) + 1$.

        We put $\theta(s) = \Theta'$ and use fresh numbers to extend $\widehat{\mathcal{M}}[s-1]$ to  $\widehat{\mathcal{M}}[s] \cong \mathcal{M}[e,\theta(s)]$. We also extend $c[s-1]$ to an isomorphism $c[s]$ from $\mathcal{M}[e,\theta(s)]$ onto $\widehat{\mathcal{M}}[s]$.
    \end{enumerate}

\subparagraph*{Construction.} 
    Following the strategy described above, we satisfy our requirements $\mathcal{R}_e$, $e\in\mathbb{N}$, one by one: $\mathcal{R}_0,\mathcal{R}_1,\mathcal{R}_2,\dots$.

    If at the end of a stage $s$ we need  to move from a requirement $\mathcal{R}_e$ to the next requirement $\mathcal{R}_{e+1}$, then (if needed) we use fresh numbers to extend the current mainland $\widehat{\mathcal{M}}[s]$ (that  looks like $\mathcal{M}[e,\theta(s)]$) to an isomorphic copy of $\mathcal{M}[e+1,\theta(s)]$. We also extend the isomorphism $c[s]$ in appropriate way.

\subparagraph*{Verification.} 
    It is clear that the constructed structure $\mathcal{A} = \bigcup_{s\in\mathbb{N}}\mathcal{A}_s$ has domain $\mathbb{N}$. Notice that $x\in\mathcal{A}_x$ for every $x\in\mathbb{N}$.
    
    Since the described construction never uses unbounded search procedures, the structure $\mathcal{A} = (\mathbb{N},\Succ^{\mathcal{A}})$ is punctual. In addition, every number $x\in\mathbb{N}$ eventually joins the mainland $\widehat{\mathcal{M}}[s]$, and by the end of each stage $s$ the structure $\widehat{\mathcal{M}}[s]$ is isomorphic to the signature-reduct $\mathcal{M}[e(s),\theta(s)]\upharpoonright \{\Succ\}$. Hence, we have
    $
    \mathcal{A} = \bigcup_{s\in\mathbb{N}} \widehat{\mathcal{M}}[s] \cong (\mathbb{N},\Succ)$.

    The ordering $\leq^{\mathcal{A}}$ is primitive recursive. Indeed, this is witnessed by the following primitive recursive procedure. Given numbers $x\neq y$, consider the value $s = \max(x,y)$ and the finite structure $\mathcal{A}_{s}$:
    \begin{itemize}
        \item If both $x,y$ belong to the mainland $\widehat{\mathcal{M}}[s]$, then we compute $c^{-1}(x)$ and $c^{-1}(y)$. Here we have
        $x <^{\mathcal{A}} y\ \Longleftrightarrow\ c^{-1}(x) < c^{-1}(y)$.
        
        \item If both $x,y$ belong to the archipelago $\operatorname{Ar}[s]$, then we have the inequality $x <^{\mathcal{A}} y$ if and only if this inequality has been explicitly declared by the \texttt{Extend} procedure by the end of the stage $s$.

        \item Otherwise, one of the numbers, say, $x$ belongs to $\widehat{\mathcal{M}}[s]$, and the other number $y$ belongs to $\operatorname{Ar}[s]$. Then we have $x <^{\mathcal{A}} y$.
    \end{itemize}


    For every function $f\in\mathcal{F}$, there exists an index $e$ such that $f(x) = t_e(x)$. We show that the function $f^{\mathcal{A}}(x) = t_e^{\mathcal{A}}(x)$ is primitive recursive. 
    
    We non-uniformly fix the stage $s^{\ast}\geq 1$ such that our construction starts working with the requirement $\mathcal{R}_e$ at the stage $s^{\ast}$. Given $x\in\mathbb{N}$, we describe a primitive recursive procedure for computing $t^{\mathcal{A}}_e(x)$. We consider $s = \max(s^{\ast},x)$ and the finite structures $\mathcal{A}_s \subset \mathcal{A}_{s+1}$. Note that $x \in \mathcal{A}_s$.
    \begin{itemize}
        \item If $x$ belongs to the mainland $\widehat{\mathcal{M}}[s]$, then by using the function $c[s+1]$ we recover the value
        $c(t_e(c^{-1}(x)))$
        that is equal to $t^{\mathcal{A}}_e(x)$.

        \item Otherwise, $x$ belongs to the archipelago $\operatorname{Ar}[s]$, and by the end of the stage $s+1$, the \texttt{Extend} procedure has explicitly declared the value $t^{\mathcal{A}}_e(x)$.
    \end{itemize}
    We obtain that for every $f\in\mathcal{F}$, the function $f^{\mathcal{A}}$ is primitive recursive. 
    
     In the description of the strategy for a requirement $\mathcal{R}_e$, we have shown that $\mathcal{R}_e$ will be eventually satisfied. Hence, the function $\Pred^{\mathcal{A}}(x)$ is not primitive recursive.
    Therefore, we conclude that the family $\mathcal{F}$ is not a basis for punctual standardness. 
    Theorem~\ref{theo:nonstandardness-meta} is proved.
\end{proof}

\subsection{Corollaries of Theorem~\ref{theo:nonstandardness-meta}}

An analysis of the proof of Theorem~\ref{theo:nonstandardness-meta} allows us to obtain the following technical result:

\begin{corollary}\label{corol:non-basis-meta}
    Suppose that an $L$-family $(\mathcal{Q},\mathcal{F})$ has a set of normal forms. Let $\mathcal{U}=\{ u_j(x_1,x_2,\dots,x_{m_j})\}_{j\in\mathbb{N}}$ be a uniform family of primitive recursive functions such that $\mathcal{U} \subseteq \mathcal{Q}$ and $\mathcal{U}$ satisfies the following property:
    \begin{description}             \item[\textnormal{($\ddagger$)}] Consider an arbitrary $j\in\mathbb{N}$ and an arbitrary tuple $\vec{r} = (r_1(x),r_2(x),\dots,$ $r_{m_j}(x))$ such that each $r_i(x)$ is either a function from $\mathcal{F}$, or a constant from $\mathbb{N}$. Then the function 
        \[
            u(x) = u_j(r_1(x),r_2(x),\dots, r_{m_j}(x))
        \]
        either belongs to $\mathcal{F}$ or is a constant, and one can (uniformly in $j$ and $\vec{r}$) primitively recursively find an expression $t'(x)$ (which is either a term $t(x)$ in the signature $\mathcal{Q}$, or a constant $d\in\mathbb{N}$) that defines $u(x)$.
    \end{description}
    Then the family $\mathcal{F} \cup \mathcal{U}$ is not a basis for punctual standardness.
\end{corollary}
\begin{proof}[Proof Sketch]
    Essentially, one needs to `incorporate' all functions $u_j\in\mathcal{U}$ into the construction of the structure $\mathcal{A}$ given in Theorem~\ref{theo:nonstandardness-meta}: that is, we have to additionally ensure that for every $j\in\mathbb{N}$, the function $u_j^{\mathcal{A}}$ is primitive recursive.

    The key construction modifications include the following.

    (1)\ The domain of the finite structure $\mathcal{M}[e,s+1]$ must include all numbers $a$ such that $a \leq u_j(b_1,b_2,\dots,b_{m_j})$ for some $j\leq e$ and $b_1,b_2,\dots,b_{m_j}\in\operatorname{dom}(\mathcal{M}[e,s])$. Informally speaking, this convention ensures that one can `prompt\-ly' compute the function $u_j^{\mathcal{A}}$ on the mainland part of the constructed structure~$\mathcal{A}$.

    (2)\ We have to modify the procedure $\texttt{Extend}(e,s)$ that grows the current archipelago. When working with a label $\boxed{q_z(w)}$ for $z\in \operatorname{Ar}[s-1]$, we need to apply to it not only the functions $f\in\sigma_e$, but also all functions 
    $f'(x) = u_j(r_1(x),r_2(x),\dots, r_{m_j}(x))$
    from Property~($\ddagger$), where $j\leq e$, and the parameters $r_i(x)$ are chosen as:
    \begin{itemize}
        \item either a constant from the current mainland $\widehat{\mathcal{M}}[s]$,

        \item or the function $q_{a}(x)$, for some label $\boxed{q_{a}(w)}$ in the archipelago $\operatorname{Ar}[s-1]$.
    \end{itemize}
    Informally speaking, this allows us to promptly compute the value $u_j^{\mathcal{A}}(b_1,b_2,\dots,$ $b_{m_j})$ in the case when some $b_i$ first appears in the archipelago part of $\mathcal{A}$.

    The rest of the proof proceeds similarly to that of Theorem~\ref{theo:nonstandardness-meta}, mutatis mutandis.
\end{proof}

We illustrate how one can apply \cref{corol:non-basis-meta}. Recall

\corArithmetic*

In order to prove Corollary~\ref{corollary_arithmetic}, we choose the following objects:
\begin{itemize}
    \item the family of primitive recursive functions $\mathcal{Q}$ is equal to $\{S,+,\times\}$,
    
    \item the family of unary functions $\mathcal{F}$ is chosen as the family of all non-constant polynomials $p(x)$ with coefficients taken from $\mathbb{N}$.
\end{itemize}
It is straightforward to check that the pair $(\mathcal{Q},\mathcal{F})$ is an $L$-family. In addition, the family $\mathcal{U} := \mathcal{Q}$ has Property~($\ddagger$) from Corollary~\ref{corol:non-basis-meta}.

Now we need to introduce a set of normal forms $T$ (recall Definition~\ref{def:norm-form}) for $(\mathcal{Q},\mathcal{F})$. Given a non-constant polynomial 
\[
    t(x) = a_k x^k + a_{k-1} x^{k-1} + \dots + a_1 x + a_0,\quad a_i\in \mathbb{N},\ a_k\neq 0,
\]
we put $N(t) = p_0^{a_0} \cdot p_1^{a_1} \cdot \ldots \cdot p_{k-1}^{a_{k-1}} \cdot p_k^{a_k}$, where $p_i$ is the $i$-th prime number. We define $T = \{ N(t) : t \in \mathcal{F}\}$.

Let $q(x)\neq t(x)$ be two polynomials from $\mathcal{F}$. Then $q(x)$ eventually dominates $t(x)$ if and only if the leading coefficient of the polynomial $u(x) := q(x) - t(x)$ is positive. Suppose that $u(x)$ is equal to $b_k x^k + \dots + b_1x +b_0$, where $b_k \in \mathbb{N}^{+}$ and $b_i \in \mathbb{Z}$. Then one can choose $n_0 := k\cdot (1 + \max\{|b_j| : j\leq k\})$ as a strict $(t,q)$-domination witness. 

Using the argument above, it is easy to show that $T$ is indeed a set of normal forms for $(\mathcal{Q},\mathcal{F})$.
Therefore, by applying Corollary~\ref{corol:non-basis-meta}, we obtain Corollary~\ref{corollary_arithmetic}.

\subsection{An Application to the Levitz's Class}\label{subsect:Levitz}

Recall the notion of the Levitz's class:

\defLevitz*

Here we prove the following:

\corLevitz*

\begin{proof}
Observe that for a function $f$ from the Levitz's class $\LL$, $f(x) = 0$ is possible only if $x = 0$ or $f = 0$. Therefore, the restrictions of the functions from $\LL$ to $\Nat^+$ are precisely the functions in the class $\LL'$ from \cite{Levitz}, Section~3. In particular, they are well-ordered with respect to the (eventual) domination order $\preccurlyeq$. 

In order to apply \cref{theo:nonstandardness-meta}, we choose:
\begin{itemize}
    \item $\mathcal{Q} = \{ S, +, \times\} \cup \{n^x : n\in\mathbb{N}^{+}\}$,

    \item $\mathcal{F}$ is the family of all non-constant functions from the class $\mathcal{L}_c$.
\end{itemize}
It is clear that every function from $\mathcal{F}$ is strictly increasing. We deduce that $(\mathcal{Q},\mathcal{F})$ is an $L$-family.
Now we need to introduce a set of normal forms $T$ for $(\mathcal{Q},\mathcal{F})$.

We assume the following encoding of the terms $f \in \LL$:
\begin{gather*}
\code{0} = 0, \;\;\; \code{1} = 1, \;\;\; \code{f+g} = 2^{\code{f}+1}\cdot 3^{\code{g}+1},\\
\code{f\cdot g} = 5^{\code{f}+1}\cdot 7^{\code{g}+1},\;\;\; \code{x^f} = 11^{\code{f}+1}, \;\;\; \code{n^f} = 13^n \cdot 17^{\code{f}+1}.
\end{gather*}
Note that the relations $f = 0$ and $f = 1$ are primitive recursive in $\code{f}$. 

We define 
\begin{equation}\label{equ:T-hat}
    T = \{ \code{f} : f\in \mathcal{F}\} \text{ and } \widehat{T} = \{ \code{f} : f\in \mathcal{L}\}
\end{equation}
Observe that both $T$ and $\widehat{T}$ are primitive recursive sets. 

In the rest of the proof, we establish that the larger set $\widehat{T}$ satisfies Condition~2 of Definition~\ref{def:norm-form} for the pair $(\widehat{\mathcal{Q}},\mathcal{L})$, where $\widehat{\mathcal{Q}} = \mathcal{Q} \cup \{ x^y\}$. This fact implies that $T$ is a set of normal forms for $(\mathcal{Q},\mathcal{F})$. 

We note here that we cannot directly apply \cref{theo:nonstandardness-meta} to the `big' pair $(\widehat{\mathcal{Q}},\mathcal{L})$, since the Levitz's class $\mathcal{L}$ is \emph{not} closed with respect to composition.

We introduce the background that is necessary for an analysis of $\widehat{T}$ and $\LL$. Following \cite{Levitz}, $f \neq 0$ is called an \emph{additive prime}, if $f = g+h$ implies $g = 0$ or $h = 0$.

A \emph{multiplicative normal form} (MNF) for $f \neq 1$ is a representation
$f = u_1^{f_1}u_2^{f_2} \ldots u_k^{f_k}$, where:
\begin{enumerate}[label=\arabic*.]
  \item each $f_i$ is additive prime;
  \item each $u_i$ belongs to $\Nat\setminus\{0,1\} \cup \{x\}$;
  \item if $i \neq j$ and $u_i,u_j \in \Nat$, then $f_i \neq f_j$;
  \item if $u_i \in \Nat$, then $f_i \neq 1$;
  \item $u_k^{f_k} \preccurlyeq \ldots \preccurlyeq u_2^{f_2} \preccurlyeq u_1^{f_1}$.
\end{enumerate}
By omission of Items 3. and 5. we obtain the notion of a \emph{set-MNF} for $f$.

Every additive prime $f \neq 1$ has a unique MNF, which will be denoted as $MNF(f)$. We assume that $MNF(1) = 1$.

Theorem 3.2 in \cite{Levitz} shows that for $f, g$ in MNF, $f = u_1^{f_1}u_2^{f_2} \ldots u_k^{f_k}$ and $g = v_1^{g_1}v_2^{g_2} \ldots v_\ell^{g_\ell}$,
\begin{align}\label{CompareMNF}
  f \prec g \;\text{ iff }\; \exists i \;\; [\forall j < i &\; (u_j^{f_j} = v_j^{g_j}) \;\&\; \nonumber \\
                                   &(i = k+1 \leq \ell \;\vee\; f_i \prec g_i \;\vee\; (f_i = g_i \;\&\; u_i \prec v_i)) ].
\end{align}

An \emph{additive normal form} (ANF) for $f \neq 0$ is a representation
$f = p_1 + p_2 + \ldots + p_k$, where:
\begin{enumerate}[label=\arabic*.]
 \item each $p_i$ is an additive prime;
 \item $p_k \preccurlyeq \ldots \preccurlyeq p_2 \preccurlyeq p_1$.
\end{enumerate}
By relaxing the requirement 2. we obtain the notion of a \emph{set-ANF} for $f$.

Every $f \neq 0$ has a unique ANF, denoted as $ANF(f)$. We put $ANF(0) = 0$.

Theorem 3.3 in \cite{Levitz} shows that for $f,g$ in ANF, $f = p_1 + p_2 + \ldots p_k$ and $g = q_1 + q_2 + \ldots q_\ell$,
\begin{equation}\label{CompareANF}
  f \prec g \;\text{ iff }\; \exists i \;\; [\forall j < i \;(p_j = q_j) \;\;\&\;\; (i = k+1 \leq \ell \;\vee\; p_i \prec q_i) ].
\end{equation}

As a corollary, any $f = u_1^{f_1}u_2^{f_2} \ldots u_k^{f_k}$ in MNF is an additive prime.
In particular, the product of additive primes is also an additive prime.

Intuitively speaking, the next series of technical lemmas establishes Condition~2.(a) of Definition~\ref{def:norm-form}: the ordering of eventual domination $f\preccurlyeq g$ is primitive recursive in the corresponding codes $\code{f},\code{g}$.

\begin{lemma}\label{setnormal} There exist functions $\bm{sa},\bm{sm} \colon \LL \to \LL$, such that: 
\begin{enumerate}[label=\arabic*.]
\item for any $f \in \LL$, $\bm{sa}(f) = p_1 + p_2 + \ldots p_k$ is a set-ANF for $f$, such that $\code{p_i} \leq \code{f}$ for all~$i$; 
\item for any additive prime $f \in \LL$, $\bm{sm}(f) = u_1^{f_1}u_2^{f_2} \ldots u_k^{f_k}$ is a set-MNF for $f$, such that $\code{f_i} < \code{f}$ for all $i$. 
\end{enumerate}
Moreover, $\code{\bm{sa}(f)}$ and $\code{\bm{sm}(f)}$ are primitive recursive in $\code{f}$.
\end{lemma}
\begin{proof} By induction on the construction of $f$. For $f \in \{0,1\}$, we take $\bm{sa}(f) = \bm{sm}(f) = f$.

Now let us suppose we have defined $\bm{sa}(f') = p_1 + p_2 + \ldots + p_k$ and $\bm{sa}(f'') = q_1 + q_2 + \ldots + q_\ell$
and also $\bm{sm}(f')$, $\bm{sm}(f'')$.

For $f = f' + f''$, we define $\bm{sa}(f) = p_1 + p_2 + \ldots + p_k + q_1 + q_2 + \ldots + q_\ell$. Obviously, $\code{p_i} \leq \code{f'} < \code{f}$
and $\code{q_j} \leq \code{f''} < \code{f}$ for all $i,j$.

If $f = f' + f''$ is an additive prime, then $f' = 0$ ($k = 0$) and we define $\bm{sm}(f) = \bm{sm}(f'')$,
or $f'' = 0$ ($\ell = 0$) and we define $\bm{sm}(f) = \bm{sm}(f')$. The condition on the codes of the exponents is clearly satisfied.

For $f = f' \cdot f''$, we define $\bm{sa}(f) = \sum_{i=1}^k \sum_{j=1}^\ell p_i \cdot q_j$. Since $p_i \cdot q_j$ is an additive prime, the definition is correct.
For all $i,j$ we have:
$$ \code{p_i\cdot q_j} \;=\; 5^{\code{p_i}+1}\cdot 7^{\code{q_j}+1} \;\leq\; 5^{\code{f'}+1}\cdot 7^{\code{f''}+1} \;=\; \code{f}. $$

If $f = f' \cdot f''$ is an additive prime, then $f'$ and $f''$ are also additive primes ($k = \ell = 1$) and we define $\bm{sm}(f) = \bm{sm}(f') \cdot \bm{sm}(f'')$.
Clearly, the product of two set-MNFs is again a set-MNF. For each factor $u^p$ of $\bm{sm}(f)$ we have $\code{p} < \max(\code{f'},\code{f''}) \leq \code{f}$.

For $f = x^{f'}$, we have $f = x^{p_1} \cdot x^{p_2} \cdot \ldots \cdot x^{p_k}$. This implies that $f$ is an additive prime, therefore we may take $\bm{sa}(f) = f$ and $\bm{sm}(f) = x^{p_1} \cdot x^{p_2} \cdot \ldots \cdot x^{p_k}$. Of course, we have $\code{p_i} \leq \code{f'} < \code{f}$ for all $i$.

Lastly, let $f = n^{f'}$. If $n = 1$, we define $\bm{sa}(f) = \bm{sm}(f) = 1$. For $n \geq 2$, we have $f = n^{p_1} \cdot n^{p_2} \cdot \ldots \cdot n^{p_k}$.
If $p_i \neq 1$, then $n^{p_i}$ is an additive prime, but we may have $p_i = 1$ for some $i$. To this end, let $I = \{i\;|\;p_i = 1\}$ and $p = n^{\sum_{i \notin I}p_i}$.
Then $p$ is an additive prime, since it is a product of additive primes. We take 
\[
\bm{sa}(f) = \underbrace{p+p+\ldots+p}_{n^{|I|} \; times}. 
\]
Observe that $\code{p} \leq \code{n^{f'}} = \code{f}$.

If $f$ is an additive prime ($I = \emptyset$), then we define $\bm{sm}(f) = \prod_{i=1}^k n^{p_i}$. The exponents $p_i$ in this definition obviously satisfy $\code{p_i} \leq \code{f'} < \code{f}$.

In all cases, if $f$ is not an additive prime, we define $\bm{sm}(f) = 0$.
\end{proof}

\begin{lemma}\label{multnormal} For any additive prime $f$, $\code{MNF(f)}$ is primitive recursive in $\code{f}$. For any additive primes $g, f$, the relation $g \preccurlyeq f$ is primitive recursive in $\code{g}, \code{f}$.
\end{lemma}
\begin{proof} We use simultaneous induction on $\code{f}$ and for all $g$, such that $\code{g} < \code{f}$.
  
  Let $f$ be an additive prime. Using Lemma \ref{setnormal} we obtain the set-MNF $\bm{sm}(f) = u_1^{f_1}u_2^{f_2} \ldots u_k^{f_k}$,
  where $\code{f_i} < \code{f}$ for all $i$. In order to produce $MNF(f)$:
  \begin{enumerate}
    \item compute $MNF(f_1), MNF(f_2), \ldots, MNF(f_k)$;
    \item enforce 3. in the definition of $MNF$ by combining factors of the form $n^{f_i}$ and $m^{f_i}$ into one factor $(nm)^{f_i}$;
    \item enforce 5. in the definition of $MNF$ by sorting the factors with respect to $\preccurlyeq$.
  \end{enumerate}
  To compare $u_i^{f_i}$ and $u_j^{f_j}$ with respect to $\preccurlyeq$, we can use Eq.~(\ref{CompareMNF}) and the inductive hypothesis applied to $f_i$ and $f_j$.
  Clearly, the whole algorithm works primitively recursively on the codes.

  Now let $g, f$ be additive primes such that $\code{g} < \code{f}$. We compute $MNF(f)$ and $MNF(g)$ using the above algorithm.
  Then we can compare them using Eq.~(\ref{CompareMNF}) and the fact that all exponents in both MNFs have smaller codes than $\code{f}$.
\end{proof}

\begin{lemma}\label{addnormal} For any $f \in \LL$, $\code{ANF(f)}$ is primitive recursive in $\code{f}$. The relation $g \preccurlyeq f$ is primitive recursive in $\code{g}, \code{f}$. Consequently, the set $\widehat{T}$ from Eq.~(\ref{equ:T-hat}) satisfies Condition~2.(a) of Definition~\ref{def:norm-form} for the pair $(\widehat{\mathcal{Q}},\mathcal{L})$.
\end{lemma}
\begin{proof} Given $f \in \LL$, first produce the set-ANF $\bm{sa}(f) = p_1 + p_2 + \ldots + p_k$ using Lemma \ref{setnormal}.
  Then compute $MNF(p_1), MNF(p_2), \ldots, MNF(p_k)$ and sort them with respect to $\preccurlyeq$ using Lemma \ref{multnormal}.
  This algorithm produces $\code{ANF(f)}$ primitive recursively from $\code{f}$.

  Now given $f, g \in \LL$, first compute $ANF(f)$ and $ANF(g)$ using the above algorithm and then compare them by application of (\ref{CompareANF}) and the second part of Lemma \ref{setnormal} for comparing the summands. This algorithm is again primitive recursive on the codes.
\end{proof}

Now we need to establish Condition~2.(b) of Definition~\ref{def:norm-form}: that is, we need to show how to primitively recursively find appropriate strict domination witnesses.

We assume all logarithms below are binary.

\begin{lemma}\label{prconv}
  For any $k \geq 5$ it holds that: \; $x \geq k^2 \;\Rightarrow\; \frac{\log x}{x} < \frac{1}{k}$.
\end{lemma}
\begin{proof}
  The function $\frac{\log x}{x}$ is strictly decreasing for $x \geq 3$. Therefore, $k \geq 5$ and $x \geq k^2$ imply:
  $$ \frac{\log x}{x} \;\leq\; \frac{\log k^2}{k^2} \;<\; \frac{1}{k}, $$
  since the last inequality is equivalent to $k^2 < 2^k$, which is true for any $k \geq 5$. 
\end{proof}

\begin{lemma}\label{preps}
  For any $u, v \in \mathbb{N}$, such that $2 \leq u < v$: \; $ \frac{1}{u(v^2-1)} \;<\; \frac{1}{2}\left(1 - \frac{\log u}{\log v}\right)$.
\end{lemma}
\begin{proof} For any $n > 0$ we have:
$$
  \frac{1}{2}\left(1 - \frac{\log u}{\log v}\right) > \frac{1}{n} \;\;\Leftrightarrow\;\; n - \frac{\log u}{\log v}n > 2
  \;\;\Leftrightarrow\;\; \left(\frac{v}{u}\right)^n > v^2.$$
By a well-known elementary inequality: $\left(\frac{v}{u}\right)^n \geq 1 + n(\frac{v}{u}-1)$.
We also have that $1 + n(\frac{v}{u}-1) \geq v^2 \;\Leftrightarrow\; n \geq \frac{u(v^2-1)}{v-u}.$
The last inequality is clearly true for $n = u(v^2-1)$, since $v-u \geq 1$.
\end{proof}

By simultaneous primitive recursion on $\max(\code{MNF(f)}, \code{MNF(g)})$,
we will define the binary functions $\omega, \nu : \LL^2 \to \mathbb{N}$, such that:
\begin{description}
\item[($\ast$)] for all additive primes $f,g$, if $f \prec g$, then $\omega(f,g)$ is a strict $(f,g)$-domination witness and $\nu(f,g)$ is a $(f\cdot x,g)$-domination witness.
\end{description}

If $f = g = 1$, we define $\omega(f,g) = \nu(f,g) = 0$.

Now let $MNF(f) = u_1^{f_1}u_2^{f_2} \ldots u_k^{f_k}$ and $MNF(g) = v_1^{g_1}v_2^{g_2} \ldots v_\ell^{g_\ell}$.

In order to define $\omega(f,g)$ we follow the cases in the first part of the proof of Theorem 3.2 in \cite{Levitz}.

\begin{description}
\item[Case 1:] $f_1 \prec g_1$. For each $i \geq 2$ we have $u_i^{f_i} \preceq u_1^{f_1}$ and in case $u_i^{f_i} \prec u_1^{f_1}$, $\omega(u_i^{f_i},u_1^{f_1})$ is the corresponding strict domination witness. Moreover, $f_1 \prec g_1$ implies that $f_1 x\preceq g_1$ and $\nu(f_1,g_1)$ is the domination witness. For any $x$ greater than all aforementioned witnesses, we have:
$$ f \preceq u_1^{kf_1}, \;\;\; \frac{\log f}{\log v_1^{g_1}} \;\leq\; \frac{k\log x}{x}\left(\frac{1}{\log v_1}\right). $$
By Lemma \ref{prconv} we have that:
$$ x \geq (k+5)^2 \;\Rightarrow\; \frac{\log x}{x} < \frac{1}{k+5} \;\Rightarrow\; \frac{k\log x}{x}\left(\frac{1}{\log v_1}\right) < 1, $$
therefore we may define 
$$ \omega(f,g) = \max(\omega(u_2^{f_2},u_1^{f_1}),\ldots,\omega(u_k^{f_k},u_1^{f_1}), \nu(f_1,g_1), (k+5)^2). $$

\item[Case 2:] $f_1 = g_1 \;\&\; u_1 \prec v_1$. For each $i \geq 2$ we have $f_i \prec f_1$ and let $\nu(f_i,f_1)$ be the $(f_ix,f_1)$-domination witness. We have to analyze the summands in: $ \frac{\log f}{\log v_1^{g_1}} \;=\; \sum_{i=1}^k \frac{f_i\log u_i}{f_1\log v_1}$.
Let us take $\epsilon = \frac{1}{2}\left(1 - \frac{\log u_1}{\log(u_1+1)}\right)$.

In both cases $v_1 \in \mathbb{N}$ and $v_1 = x$ we have:
$$ x > u_1 \;\Rightarrow\; \frac{\log u_1}{\log v_1} \;\leq\; \frac{\log u_1}{\log (u_1+1)}. $$

For every $i \geq 2$ and $x \geq \nu(f_i,f_1),\; x \geq u_i$ we have:
$$ \frac{f_i\log u_i}{f_1\log v_1} \;\leq\; \frac{\log u_i}{x \log v_1} \;\leq\; \frac{\log x}{x}. $$
By Lemma \ref{preps}: $\frac{1}{u_1((u_1+1)^2 - 1)} < \epsilon$. Therefore by Lemma \ref{prconv}:
$$ x \geq \left(ku_1((u_1+1)^2 - 1)+5\right)^2 \;\Rightarrow\; \frac{\log x}{x} \;<\; \frac{1}{ku_1((u_1+1)^2 - 1)+5} \;<\; \frac{\epsilon}{k}. $$
By summing for $i = 2, \ldots, k$ we obtain: $\sum_{i=2}^k \frac{f_i\log u_i}{f_1\log v_1} \;<\; \epsilon$. Finally, by the choice of $\epsilon$:
$$ \sum_{i=1}^k \frac{f_i\log u_i}{f_1 \log v_1} \;=\; \frac{\log u_1}{\log v_1} + \sum_{i=2}^k \frac{f_i\log u_i}{f_1 \log v_1} \;<\; \frac{\log u_1}{\log (u_1+1)} + \epsilon \;<\; 1. $$
We define:
$$ \omega(f,g) = \max(\nu(f_2,f_1),\ldots,\nu(f_k,f_1),u_1+1,u_{max}, \left(ku_1((u_1+1)^2 - 1)+5\right)^2), $$
where $u_{max}$ is the supremum of the set $\{u_i\;|\;2 \leq i \leq k \;\&\; u_i \neq x\}$.

\item[Case 3:] $f_1 = g_1 \;\&\; u_1 = v_1$. Obviously, we define:
$$ \omega(f,g) = \omega(u_2^{f_2} \ldots u_k^{f_k},v_2^{g_2} \ldots v_\ell^{g_\ell}). $$
All other cases are symmetric (with reversed roles of $f$ and $g$).
\end{description}

Now to define $\nu(f,g)$ we follow the third part of the proof of Theorem 3.2 in \cite{Levitz}. We may assume $f \prec g$ (the case $g \prec f$ will be symmetric and we define $\nu(f,f) = 0$).

\begin{description}
\item[Case 1:] $f = 1$. Then $x \preceq g$ and we take $\nu(f,g) = 0$.

\item[Case 2:] $f \neq 1$. Then $x \preceq f$. Let $w$ be the common part of $MNF(f)$ and $MNF(g)$.
Thus we have $f = f'w$ and $g = g'w$, where $f'$ and $g'$ are additive primes, which do not have common terms in their MNF.
We will show that $f'x \preceq g'$ and from there we will extract the $(fx,g)$-domination witness.


\item[Case 2.1:] $f' = 1$. Then $g' \neq 1$, therefore $x \preceq g'$. We define $\nu(f,g) = 0$.

\item[Case 2.2:] $f' \neq 1$. We have $MNF(f') = s_1^{t_1}s_2^{t_2} \ldots s_m^{t_m}$ for $m \geq 1$
and $MNF(g') = r_1^{q_1}r_2^{q_2}\ldots r_p^{q_p}$ for $p \geq 1$. We compare $s_1^{t_1}$ and $r_1^{q_1}$.

\item[Case 2.2.1:] $t_1 \prec q_1$. Then Case 1 in the first part shows that $f'x \prec g'$,
that is $s_1^{t_1}s_2^{t_2} \ldots s_m^{t_m}x \prec r_1^{q_1}r_2^{q_2}\ldots r_p^{q_p}$. We define:
\begin{align*}
 \nu(f,g) &= \omega(s_1^{t_1}s_2^{t_2} \ldots s_m^{t_m}x,r_1^{q_1}r_2^{q_2}\ldots r_p^{q_p})\\
          &= \max(\omega(s_2^{t_2},s_1^{t_1}),\ldots,\omega(s_m^{t_m},s_1^{t_1}), \nu(t_1,q_1), (m+6)^2).
\end{align*}

\item[Case 2.2.2:] $t_1 = q_1 \;\&\; s_1 \prec r_1$. Then Case 2 in the first part shows that $f'x \prec g'$,
that is $s_1^{t_1}s_2^{t_2} \ldots s_m^{t_m}x \prec r_1^{q_1}r_2^{q_2}\ldots r_p^{q_p}$. We define:
\begin{align*}
 \nu(f,g) &= \omega(s_1^{t_1}s_2^{t_2} \ldots s_m^{t_m}x,r_1^{q_1}r_2^{q_2}\ldots r_p^{q_p})\\
          &= \max(\nu(t_2,t_1),\ldots,\nu(t_m,t_1),s_1+1,s_{max}, \left((m+1)s_1((s_1+1)^2 - 1)+5\right)^2).
\end{align*}
The case $t_1 = q_1 \;\&\; s_1 = r_1$ is by assumption not possible.
\end{description}

And finally, we will define $d : \LL^2 \to \mathbb{N}$, such that for any $f,g\in \LL$,
$f \prec g$ implies that $d(f,g)$ is a strict $(f,g)$-domination witness.

Given $f, g \in \LL$, we first compute $ANF(f) = p_1 + p_2 + \ldots p_k$ and $ANF(g) = q_1 + q_2 + \ldots + q_m$. If $ANF(f)$ is a prefix of $ANF(g)$, we define $d(f,g) = 0$.
Otherwise, let $i \leq k$ be the first index, such that $p_i \neq q_i$.
\begin{description}
\item[Case 1:] $p_i \prec q_i$. Then $p_i x \;\leq\; q_i$ for $x \geq \nu(p_i,q_i)$. We have $p_i + \ldots + p_k \;\leq\; k\cdot p_i \leq p_ix \leq q_i + \ldots + q_k$ for $x \geq k$. We define $d(f,g) = \max(\nu(p_i,q_i),k)$.

\item[Case 2:] $q_i \prec p_i$. Symmetric, $d(f,g) = \max(\nu(q_i,p_i),m)$.
\end{description}

Of course, $d$ is primitive recursive on the codes. We conclude that the set $\widehat{T}$ from Eq.~(\ref{equ:T-hat}) satisfies Definition~\ref{def:norm-form} for the pair $(\widehat{\mathcal{Q}},\mathcal{L})$. This implies that $T$ is a set of normal forms for the $L$-family $(\mathcal{Q},\mathcal{F})$. By \cref{theo:nonstandardness-meta}, each of the sets $\mathcal{F}$ and $\mathcal{L}_c = \mathcal{F} \cup \{\text{the constant functions}\}$ is not a basis for punctual standardness.
Corollary~\ref{corollary_Levitz} is proved.
\end{proof}

\section{Sufficient Condition for Punctual Standardness}\label{sec contribution 2 technical}

In this section we establish sufficient conditions for punctual standardness, building on the results concerning punctual categoricity from \cite{kalimullin_algebraic_2017}. We recall that $(\psi_e)_{e\in\mathbb{N}}$ is a computable enumeration of the class of unary primitive recursive functions. For convenience we write $Nf(e,x,y,s)$ for the elementary relation that holds precisely when $\psi_{e,s}(x)\!\downarrow$ and $\psi_e(x) = y$. 

For the moment, we restrict the language to one constant and two unary functions. Let $\mathcal{B}_n$ be the $n$-th punctual structure in this language, defined by $\mathcal{B}_n = (\mathbb{N}, o_n, s_n, c_n)$, where $n$ is the code of a triple, $o_n = (n)_0, s_n = \psi_{(n)_1}, c_n = \psi_{(n)_2}$.

When counting the number of steps in the two lemmas below, we consider only the parts of computation that invoke the relation $Nf$.

\begin{lemma} The relation $Cyc$ is elementary, where $Cyc(n,y,l,s)$ holds precisely when $y$ has $c_n$-cycle length $l$ and this verification requires $\leq s$ steps.
\end{lemma}
\begin{proof}
  Informally, $Cyc(n,y,l,s)$ holds when the sequence $y, c_n(y), c_n^2(y), \ldots, c_n^{l-1}(y)$ consists of distinct elements, satisfies $c_n^l(y) = y$, and all computations take $\leq s$ steps.

  Therefore, $Cyc(n,y,l,s)$ holds if and only if there exists a list $L = \langle L_0,L_1,\ldots, L_l\rangle$ such that: (1) $L_0 = L_l = y$; (2) $\forall i < l$, $Nf((n)_2, L_i, L_{i+1}, s)$; (3) $\forall i < j < l$, $L_i \neq L_j$.
  
  Moreover, the code of $L$ is bounded by $\langle y, s, \ldots, s \rangle$ of length $l+1$, which is elementary in $y,s,l$.
\end{proof}

\begin{lemma} The relation $PatCyc$ is elementary, where $PatCyc(n,i,x,s)$ holds precisely when the elements
  $$ y_0 = s_n^{i}(x),\;\; y_1 = s_n^{2^{n+1}}(y_0),\;\; \ldots,\;\; y_x = s_n^{2^{n+1}}(y_{x-1}),\;\; y_{x+1} = s_n^{2^{n+1}}(y_x)$$
  have $c_n$-cycle lengths $2n+1,2n+2,\ldots,2n+2,2n+1$, respectively, and these computations take \emph{exactly} $s$ steps.
\end{lemma}
\begin{proof} We define $PatCyc'(n,i,x,s)$ with the same meaning, but for $\leq s$ steps and then we set
  $$PatCyc(n,i,x,s) \;\;\Leftrightarrow\;\; PatCyc'(n,i,x,s) \;\&\; \neg PatCyc'(n,i,x,s-1).$$

  Thus, $PatCyc'(n,i,x,s)$ if and only if there exists a list $L = \langle L_0,L_1,\ldots, L_l\rangle$, such that:

  \begin{enumerate}
  \item $l = i+(x+1)\cdot 2^{n+1}$,
  
  \item $L_0 = x$,
  
  \item $\forall i < l$, $Nf((n)_1, L_i, L_{i+1}, s)$,
  
  \item $Cyc(n,L_i,2n+1,s)$,

  \item $\forall j < x, \; Cyc(n,L_{i+(j+1)\cdot 2^{n+1}},2n+2,s)$,
  
  \item $Cyc(n,L_{i+(x+1)\cdot 2^{n+1}},2n+1,s)$.
  \end{enumerate}
  As in the previous lemma, the code of $L$ is bounded by $\langle x, s, \ldots, s \rangle$ of length $l+1$, which is elementary in $n,i,x,s$.
\end{proof}

\begin{theorem} The function $f$, defined in Proposition 4.2 of \cite{kalimullin_algebraic_2017}, is elementary.
\end{theorem}

\begin{proof}
  First we show that the relation $u_x = y$ is elementary in $n,x,y$,
  where $u_x$ is the number of steps up to stage $x$. If the corresponding search fails at stage $x$, we stipulate that $u_x = u_{x-1}$, with the convention $u_{-1} = 0$ (in \cite{kalimullin_algebraic_2017}, $u_x$ is left undefined in such a case).

  The relation $u_x = y$ can be decided by the following recursive algorithm:

  \begin{enumerate}
  
  \item If $x = 0$, set $y' = 0$.

  \item If $x > 0$:

    \begin{enumerate} 
      
      \item Search for $y' \leq y$, such that $u_{x-1} = y'$.

      \item In case the search is unsuccessful, return \emph{false}.

    \end{enumerate}

  \item Search for the least $i \in [1;\langle n,y'\rangle +(x+1)2^{n+1}]$
  
  \;\;\;\;\;\;\;  such that $PatCyc(n,i,x,s)$ holds for some (unique) $s \leq y$.
  
  \item If the search is successful, return $y = y'+s$.

  \item If the search is not successful, return $y = y'$.

  \end{enumerate}

  More formally, if $P(n,x,y)$ is the list of truth values $$u_x = 0,\; u_x = 1,\; \ldots,\; u_x = y,$$
  then the above algorithm computes the code of $P$ using bounded primitive recursion on $x$. Hence $P$ is elementary, and therefore the relation $u_x = y$ is also elementary.

  Observe that if the search in 3. fails for some $x$, it may still succeed for some $x' > x$. In \cite{kalimullin_algebraic_2017}, the construction does not carry out another search after the first unsuccessful attempt, but this distinction is immaterial from our perspective.

  Now, given an input $t$, we compute $f(t)$ using the following elementary algorithm:

  \begin{enumerate}
  
  \item Compute $n,m$, such that $t = \langle n,m\rangle$.

  \item If $u_0 = 0$ or $m \leq u_0$, return $2n+1$.

  \item Search for the least $x < m$ such that $m \leq u_{x+1}$ or $u_{x+1} = u_x$.

  \item Search for $v < m$ such that $u_x = v$.

  \item If $m \equiv v+1 \text{ mod } (x+2)$, return $2n+1$.

  \item If $m \not\equiv v+1 \text{ mod } (x+2)$, return $2n+2$.

  \end{enumerate}
  
  Observe that the relation $u_{x+1} = u_x$ is not elementary, but when it is used in 3. we have $u_0 < \ldots < u_x < m$, so that $u_{x+1} = u_x$ can be replaced by $\exists v' < m\; (u_x = v' \;\&\; u_{x+1} = v')$.
\end{proof}

Next we define the structure $\mathcal{A}_f = (\mathbb{N}, 0, s_K, c_K)$, where the $n$-th cycle has $f(n)$ elements, begins at $2n$, and the other elements of the cycle are filled with successive odd numbers.

\begin{theorem}
  The functions $s_K$ and $c_K$ are elementary.
\end{theorem}
\begin{proof}
  The cycles $0,1,\ldots,n$ contain $p(n) = \sum_{s=0}^n (f(s) - 1)$ odd numbers. Hence for $n > 0$, the first odd number in the $n$-th cycle is $2p(n-1)+1$.
  
  Given an input $x$, we first determine the number $n$ of the cycle that contains $x$ and then the index $i$ such that, if $x$ is odd, then $x$ is the $i$-th odd element in its cycle.

  \begin{enumerate}
  
  \item If $x$ is even, then $n = \lfloor \frac{x}{2} \rfloor$.
  
  \item If $x = 1$, then $i = 0$ and $n = 2-f(0)$.
  
  \item If $x > 1$ is odd, we search for $n < x$ and $i < x$
  
  \;\;\;\;\;\;\; such that $x = 2p(n-1) + 2i + 1$.

  \end{enumerate}

  Observe that for any $s$, we cannot have $f(s) = f(s+1) = 1$. This justifies the output in 2. and the bound $n < x$.
  
  Finally, we compute $s_K(x) = 2n+2$ and 
    $$ c_K(x) = \begin{cases}
                   x & \text{ if } x \text{ is even and } f(n) = 1,\\
                   2p(n-1) + 1 & \text{ if } x \text{ is even and } f(n) > 1,\\
                   x+2 & \text{ if } x \text{ is odd and } i+2 \neq f(n),\\
                   2n & \text{ if } x \text{ is odd and } i+2 = f(n).
    \end{cases} $$
\end{proof}

\begin{theorem} \label{thm_basis}
  The set $\{s_K, c_K\}$ is a basis for punctual standardness.  
\end{theorem}
\begin{proof}
  Fix a copy $\mathcal{A}$ such that $s_K^{\mathcal{A}}$ and $c_K^{\mathcal{A}}$ are primitive recursive. Let $c_\mathcal{A}$ be the unique isomorphism from $\mathcal{S}$ to $\mathcal{A}$. Since $s_K^\mathcal{A}$ and $c_K^\mathcal{A}$ are primitive recursive, we can choose $n$ such that $\mathcal{B}_n = (\mathbb{N},0^\mathcal{A},s_K^\mathcal{A},c_K^\mathcal{A})$.
  Clearly $\mathcal{A}_f = (\mathbb{N}, 0, s_K, c_K)$ is isomorphic to $\mathcal{B}_n$ via $c_\mathcal{A}$. By the construction of $f$, this implies that $\mathcal{B}_n$ and $\mathcal{A}_f$ are punctually isomorphic. Therefore, $c_\mathcal{A}^{-1}$ is also primitive recursive. We conclude from Observation \ref{observation} that $\mathcal{A}$ is punctually standard. 
\end{proof}

\elementarybasis*

\begin{proof}
    Fix any substitution basis for elementary functions $B$ and a copy $\mathcal{A}$ of $\mathcal{S}$ such that every function from $B$ is primitive recursive on $\mathcal{A}$. Then $s_K$ anc $c_K$ are also primitive recursive on $\mathcal{A}$ since they can be obtained from $B$ via substitutions. From Theorem \ref{thm_basis} we conclude that $\mathcal{A}$ is punctually standard.
\end{proof}

\printbibliography

\end{document}